\documentclass[final]{siamltex}
\usepackage{amsfonts, epsfig}
\normalbaselineskip

\newcommand{\LV}{\left|}
\newcommand{\RV}{\right|}
\newcommand{\LC}{\left(}
\newcommand{\RC}{\right)}
\newcommand{\LB}{\left[}
\newcommand{\RB}{\right]}
\newcommand{\LCB}{\left\{}
\newcommand{\RCB}{\right\}}

\newcommand{\D}[2]{\frac{d#1}{d#2}}
\newcommand{\paren}[1]{\left(#1\right)}
\newcommand{\mb}[1]{\mathbf{#1}}

\newcommand{\bt}{{\mathbf t}}
\newcommand{\bv}{{\mathbf v}}

\newcommand{\bq}{\mathbf q }

\newcommand{\ff}{\bf f}

\newcommand{\tr }{{\textrm {tr\,}}}

\newcommand{\bn}{{\mathbf n}}

\newcommand{\Cauchy}{\mathcal T}

\newcommand{\wfh}{\mathcal{W}_{\begin{tiny}{\hbox{FH}}\end{tiny}}}
\newcommand{\Id}{{\bf I}}

\newcommand{\bx}{{\mathbf x}}
\newcommand{\bX}{{\mathbf X}}

\newtheorem{remark}[theorem]{Remark}
\begin{document}

\title{Long-time existence of classical solutions to a 1-D swelling gel}

 \author{M.~Carme Calderer \thanks{School of Mathematics, University of Minnesota, 
206 Church Street S.E., Minneapolis, MN 55455, USA.({\tt calde014@.umn.edu}).}
\and
Robin Ming Chen\thanks{Department of  Mathematics, University of Pittsburgh, 301 Thackeray Hall, Pittsburgh, PA  15260, USA. ({\tt mingchen@pitt.edu})}}



\thanks{M.C. Carme was supported in part by NSF grants DMS-0909165. R.M. Chen was supported in part by NSF grants DMS-0908663.}

\maketitle
\begin{abstract}
In this paper we derived a model which describes the swelling dynamics of a gel and study the system in one-dimensional geometry with a free boundary. The governing equations are hyperbolic with a weakly dissipative source. Using a mass-Lagrangian formulation, the free-boundary is transformed into a fixed-boundary. We prove the existence of long time $C^1$-solutions to the transformed fixed boundary problem.
\end{abstract}


\section{Introduction}\label{sec_intro}
In this paper, we study existence of long time $C^1$-solutions of  a  free boundary problem
  modeling    swelling of gels,  in one space dimension.  The gel is  assumed to be a mathematical mixture of 
  polymer and solvent. We assume that the gel is surrounded by solvent and that its  boundary is fully permeable to solvent. 
The governing equations  consist of the laws of balance of mass and linear momentum of the components and form a weakly dissipative 
hyperbolic system.   The scaling leading to hyperbolic dynamics is motivated by the modeling of polysaccharide gels occurring in nature, for instance as 
motor devices of gliding bacteria ({\it myxobacteria}).  
In general, hypebolic dynamics is a signature property of active gels encountered in models of the cytoskeleton of living cells and also in studies of biomimetic non-equilibrium  gels \cite{Levine}, \cite{Sharma}.

We assume that a gel is a saturated, incompressible and immiscible mixture of polymer and solvent. Incompressibility in the context of mixtures refers to every component having 
constant mass density in the pure state. (In particular, it does not preclude the polymer from experiencing deformations with very large change in volume, 
as much as 500 percent, if ions are present). 
 Immiscibility corresponds to the constitutive equations depending on the volume fraction, $\phi_1$($\equiv \phi$) and $\phi_2$, respectively, of the polymer and solvent 
component. The saturation assumption (also known in some literature as {\it incompressibility}) is the statement of the volume fractions adding to 1. 
The balance laws consist of the equations of balance of mass and linear momentum of each component together with the saturation constraint. 
The coupling between component equations takes place through the Flory-Huggins energy of mixing, the drag forces and the boundary conditions. 

For gels made of entangled polymers such as those used in biomedical devices, scaling arguments show that elastic (\cite{YD1, YD2, suo}) and  dissipative effects often dominate inertia,  that enters the early dynamics only. Moreover, in \cite{CCLZ} the authors assume that the 
dissipation of both the polymer and solvent is Newtonian. 

In this paper, we consider gels made of polysacharide networks, with dissipation parameters several orders of magnitude smaller (as much as $10^{-7}$) 
than those of the entangled polymer counterparts  \cite{Zhang}. This motives us to keeping the inertia terms to study the intermediate time dynamics of 
the polysaccharide  systems (as well as the early dynamics of the highly dissipative gels). 

The energy of the system is the sum of the nonlinearly elastic stored energy function of the polymer and the Flory-Huggins energy of mixing. 
We establish a dissipation law satisfied by solutions of the governing system in arbitrary space dimension. 
There are two types of boundary conditions complementing the elasticity and the Euler equation. These are of traction-displacement kind, for the 
 gel equation (this being the sum of the elasticity and the Euler equation for the fluid), and either scalar Dirichlet or traction boundary conditions for the 
the Euler equation of the fluid. These boundary conditions, also considered in earlier works on gels \cite{Zhang}, \cite{Brandon}, are motivated 
by the works of Doi and Yamaue \cite{YD1, YD2}. They express properties of permeability of the gel interface to the surrounding fluid, and may range from 
impermeable (Dirichlet boundary) to fully permeable (traction boundary). 
However, these conditions have been found not be exact, from the point of view that they do not necessarily guarantee the balance of mass, linear momentum and 
energy at the interface between the gel and the external fluid. In fact, they differ from the balance laws at the interface by a term quadratic 
on  the normal component of the difference of the polymer and solvent velocities, and proportional to the fluid inertia.  Moreover, the assumption of viscous  interface stress  yields
 a fully dissipative energy law  \cite{calderer-chen-micek-mori2010}. 

In this work, we consider a polymer separated from the surrounding fluid and  confined to a strip domain. We assume that at  time $t=0$ the polymer enters 
into contact with the fluid at the boundary points, causing it to enter and swell the polymer. 
The requirement for the three dimensional governing system to have one-dimensional solutions determines the (pressure) Lagrange multiplier. 
Moreover, it follows from the equation of balance of mass, that the center of mass velocity is independent of space. This motivates us to rewrite 
the governing equations in terms of the relative fluid-polymer velocity,  $\mathbf U=\bv_1-\bv_2$. Another implication of the one dimensional geometry
 is that the volume fraction on the boundary is fully determined from the traction boundary condition, provided the fluid inertia on the boundary conditions 
is neglected. In this case,   a problem  originally formulated
with traction boundary becomes a of Dirichlet  type. 

Without loss of generality and to simplify the analysis, we assume that the initial domain (also used as reference configuration) to be the symmetric 
interval $(-1, 1)$. We use the mass-Lagrangian change of variables from gas dynamics, to transform the free boundary problem into one with fixed domain. We find that
$G'(\phi)+ U^2<0$ is a necessary and sufficient condition for the governing equation to be hyperbolic,  and so, it is necessary that $G'(\phi)<0$ hold. 
The function $G(\phi)$ corresponds to a one-dimensional stress, and the condition of negative  derivative  with respect to $\phi$ expresses monotonicity of
the stress with respect to the  deformation gradient $u_x=\frac{\phi_I}{\phi}$,  where $\phi_I$ denotes the reference volume fraction of the polymer. 
Heuristically, if $G'$ is bounded,  the hyperbolicity condition is satisfied for sufficiently small relative velocity. 
Using the framework introduced in \cite{BD} in dealing with the hyperbolic initial-boundary-value problem, 
we obtain the local well-posedness of strong solutions for data satisfying the hyperbolicity and uniform Kreiss-Lopatinski\u\i \ condition. We also consider the long-time existence of classical solutions. We use the method of characteristics to estimate the $C^1$-norm of solutions. Due to the lack of strong dissipation of the system and the boundary damping, our estimate does not lead to the existence of global solutions. Instead, we prove that the existence time is at least of order $O\LC |\log\varepsilon| \RC$, where $\varepsilon$ is the $C^1$-deviation from the equilibrium.

The rest of the paper is organized as follows. In Section \ref{sec_modeling} we set up the mathematical model that describes the swelling dynamics of a gel and derive the constitutive equations. We also formulate the one-dimensional system with boundary conditions. In Section \ref{sec_fix} we transform the one-dimensional free-boundary problem into a fixed-boundary problem using the mass-Lagrangian change of variables, and prove the local well-posedness of strong solutions to the transformed problem. In Section \ref{sec_global} we prove the long-time well-posedness of classical solutions.

\section{General model of gel-swelling}\label{sec_modeling}
We begin by setting up the model of the hydrogel postulated in \cite{CZ, CCLZ}. We assume that a gel is a saturated, incompressible and immiscible mixture of  elastic solid and  fluid.
 In the reference configuration, the polymer occupies a domain $\Omega\in \mathbb{R}^3$. 
 The solid undergoes a deformation according to the smooth  map
\begin{equation}
\bx=\bx(\bX, t), \quad \textrm {such that\,\,} \det(\nabla\bx)>0,\,\, \bX\in\Omega.
\end{equation}
 We let  $\Omega_t=\bx(\Omega)$ denote  the domain occupied by the gel at time $t\geq 0$, $\Gamma\subset\partial\Omega_t$ be the gel-solvent interface (in the case that $\Gamma=\partial\Omega_t$, the gel is fully surrounded by the solvent), and let $F=\nabla\bx$ denote the deformation gradient.  We label  the polymer and fluid components with indices {\it 1} and {\it 2}, respectively.
A point $\bx\in\Omega_t$ is occupied by, both, solid and fluid at volume fractions
 $\phi_1=\phi_1(\bx,t)$ and $\phi_2=\phi_2(\bx,t)$, respectively.  We let $\bv_1=\bv_1(\bx,t)$ and $\bv_2=\bv_2(\bx, t)$ denote the corresponding velocity fields.

An {immiscible} mixture is such that  the  constitutive equations  depend explicitly on the volume fractions $\phi_i$.
  We let $\rho_i$  denote the {mass} densities of the $i$th component (per unit volume of gel). 
 They are related to the {intrinsic} densities $\gamma_i$ by $\rho_i=\gamma_i \phi_i$, $i=1,2$,  which are constant in the case of  an incompressible mixture.

\subsection{Governing equations}\label{subsec_goveqns}
The assumption of saturation of the mixture, that is, that no species other than polymer and fluid are present,
is expressed by the equation
 \begin{equation}\label{sat}
 \phi_1+\phi_2 = 1.
\end{equation}
(In some terminologies, this is  condition is also known  as incompressibility). 
 The governing equations consist of the balance of mass and linear momentum of each components  as well as the chain-rule relating the time derivative to the gradient of deformation with the velocity gradient:	
\begin{eqnarray}
\frac{\partial \phi_i}{\partial t} +\nabla\cdot(\phi_i\bv_i) &&= 0, \label{mass}\\		
\gamma_i\phi_i\left(\frac{\partial \bv_i}{\partial t} + (\bv_i\cdot\
\nabla)\bv_i\right) &&= \nabla\cdot\Cauchy_i+ {\ff}_i, \quad \ff_1+\ff_2={\bf 0}, \label{lin-momentum} \\
\frac{\partial F}{\partial t} + (\bv_1\cdot\nabla)F &&=(\nabla\bv_1)F,\label{chainrule}
\end{eqnarray}
$\bx\in\Omega_t$, $i=1,2$.  Here $\ff_i$ represent drag forces between the components, and $\Cauchy_i$ is the Cauchy stress tensor of the $i$th component. Note that we have used the incompressibility assumption $\gamma_i=\textrm {constant}$ in the previous equations. 
 The Lagrangian form of the equation of balance of mass of the polymer component  is 
\begin{equation}\label{mass-lagrangian}
\phi_1\det F = \phi_I\quad\hbox{in } \Omega,
\end{equation}
where $\phi_I$ is the initial polymer volume fraction in the reference configuration. Note that if $\det F = 1,$ then $\phi_I = \phi_1,$ so that no changes in the volume fraction of the gel correspond to no changes in volume fraction.  Hence equation (\ref{mass-lagrangian}) is equivalent to equation (\ref{mass}) with $i=1.$
Adding equations (\ref{mass}), and using the constraint (\ref{sat}) gives
\begin{equation}\label{div}
 \nabla\cdot(\phi_1\bv_1+\phi_2\bv_2)=0. 
\end{equation}

 With this formulation, the final system of the governing equations are (\ref{lin-momentum}) and (\ref{chainrule}), subject to the constraints in (\ref{sat}), (\ref{mass-lagrangian}), and (\ref{div}. 
 
 \subsection{Boundary conditions}\label{subsec_bdrycond}
 
 We now specify boundary conditions at the gel-water interface $\Gamma$.  We observe that the equations of balance of mass are first order, scalar equations for $\phi_1$ and $\phi_2$ that require prescribing initial conditions only. Later, we will assume that the network part of the gel is elastic, with Newtonian dissipation. The latter will also be assumed for the in-gel and exterior fluids. This requires one vector boundary condition for each equation. 
There are two kinds of boundary conditions that we will discuss, in order to come up with a total of six scalar boundary conditions for the system:
\begin{itemize}
 \item statements of balance of mass,  linear momentum and energy across $\Gamma$, and
\item  relations expressing the permeability properties of the interface. 
\end{itemize}

\noindent
{\sl \underline{Balance of mass across $\Gamma$.}\, }
We denote $\bv_f$ the velocity field of the fluid outside the gel network, which is assumed to be Newtonian. The equation of balance of mass of the fluid across the interface $\Gamma$ is 
\begin{equation}
(\bv_f-\bv_1)\cdot\bn= \phi_2(\bv_2-\bv_1)\cdot\bn:=w, \label{normal-velocity}
\end{equation}
where $\bn$ denotes the unit normal to $\Gamma$, pointing from the gel to the pure liquid.  Note that the last relation defines the normal component, $w$, of the relative velocities.  We also let $\bt_1$  and $\bt_2$ denote a pair of orthonormal vectors perpendicular to $\bn$.  Let  $\bv $  be a vector field on  $\Gamma$.  We use the notation $\|$ to denote the vector components tangent to $\Gamma$ at a point. That is, we write 
$$\bv= (\bv\cdot\bn)\bn + \bv_{\|}=(\bv\cdot\bn)\bn + \sum_{i=1,2}(\bv\cdot\bt_i)\bt_i,$$
\medskip

\noindent
{\sl \underline{Continuity of the  components of the velocity tangent to $\Gamma$.}\,} 
We assume the following no-slip condition
\begin{equation}
 (\bv_f-\bv_1)_{\|}= (\bv_2-\bv_1)_{\|}:= \mb{q}.\label{tangent-velocity}
\end{equation} 
\begin{remark} 
The no-slip assumption equating the tangential components of the fluid velocity, relative to the polymer one, across $\Gamma$ may admit a generalization of the form  

\begin{equation}
 (\bv_f-\bv_1)_{\|}= \eta_v(\bv_2-\bv_1)_{\|}:= \mb{q},\label{tangent-velocity2}
\end{equation}
 where $\eta_v>0$ is a dimensionless parameter related to the surface viscosity. 
We will observe next that this brings tangential stress components in the balance of linear momentum equation. 
\end{remark}

\noindent
{\sl \underline{Balance of linear Momentum across $\Gamma$.}\,}
We next consider force balance across the interface $\Gamma$. Consider a point $\bx\in\Gamma$ at a time $t$, and let us observe this point in an inertial frame traveling with the same velocity as the polymer at point $\bx$ and time $t$. Given 
equations (\ref{normal-velocity}) and (\ref{tangent-velocity}) we observe that the water in the fluid region travels at 
velocity $w\bn + \bq$ and in the gel region travels with velocity $\frac{w}{\phi_2}\bn +\bq$. the change of mass of water across the interface per unit time at point $\bx$ at time $t$ is given by $\gamma_2 |w|$.  As water comes out of or into the gel region, corresponding to the collapsing ($w>0$) or swelling state ($w<0$), respectively, there is the following amount of momentum change from the gel region to the water region per unit time:
 \begin{equation}\label{momentumchange}
  \gamma_2 |w|\LC w\bn - \frac{w}{\phi_2}\bn \RC. 
 \end{equation}
 The force acting at the interface is given by:
\begin{equation}
(\Cauchy_1+  \Cauchy_2 - \Cauchy)\bn.  
\end{equation}
The law of balance of linear momentum gives
\begin{equation}\label{linear-momentum -gamma}
 (\Cauchy_1+  \Cauchy_2 - \Cauchy)\bn=   \gamma_2 |w|w\LC 1 - \frac{1}{\phi_2} \RC \bn.
\end{equation}
Note that this implies, in particular, that
\begin{equation}
 (\Cauchy_1+  \Cauchy_2 - \Cauchy)\bn\cdot \bt_i=0, \,\, i=1,2. \label{stress-tangent-gamma}
\end{equation}
\medskip
Note that the difference in the stress across the boundary is
normal to the interface. In particular, we have:
\begin{equation}
\LC{\Cauchy_1\mb{n}+\Cauchy_2\mb{n} - \Cauchy\mb{n}}\RC\cdot\mb{q}=0\label{Tq}
\end{equation}
since $\mb{q}$ is tangential to the surface $\Gamma$.

At this point, we have the boundary conditions (\ref{normal-velocity}),
(\ref{tangent-velocity}) and (\ref{linear-momentum -gamma}) on $\Gamma$, which together give us six boundary
conditions. Mass balance and total force balance would provide
the necessary number of boundary conditions if the interior
of $\Omega_t$ were composed of a one-phase medium. Here, the
interior of $\Omega_t$ is a two-phase gel. We thus require
three additional boundary conditions, assuming that all of
the phases have bulk viscous stresses. This corresponds to
specifying some condition that involves $w$ and $\mb{q}$.
The appropriate forms for these boundary conditions will be
discussed shortly.

We now check that the boundary conditions (\ref{normal-velocity}), (\ref{tangent-velocity})
and (\ref{linear-momentum -gamma}) lead to mass and momentum conservation.
Let us first check that the amount of water is conserved:
\begin{eqnarray*}
&&\D{}{t}\paren{\int_{\Omega_t} \phi_2d\mb{x}+\int_{\Omega_t^c}d\mb{x}}\\
=&&\int_\Gamma \phi_2(\mb{v}_2-\mb{v}_1)\cdot \mb{n}d\bf S
-\int_\Gamma (\mb{v}_f-\mb{v}_1)\cdot \mb{n}d\bf S\\
=&&0, 
\end{eqnarray*}
where we used (\ref{normal-velocity}) 
in the second equality.

Now, we turn to momentum conservation. We have:
\begin{eqnarray*}
&&\D{}{t}\paren{\int_{\Omega_t}
\paren{\gamma_1\phi_1\mb{v_1}+\gamma_2\phi_2\mb{v}_2}
d\mb{x}+\int_{\Omega_t^c}\gamma_2\mb{v}_fd\mb{x}}\\
=&&\int_\Gamma
\paren{\gamma_2\phi_2\LV(\mb{v}_1-\mb{v}_2)\cdot \mb{n}\RV\mb{v}_2
+\Cauchy_1\mb{n}+\Cauchy_2\mb{n}}d\bf S\\
-&&\int_\Gamma \paren{\gamma_2\LV(\mb{v}_1-\mb{v}_f)
\cdot \mb{n}\RV\mb{v}_f+\Cauchy\mb{n}}d\bf S\\
=&&\int_\Gamma \paren{\gamma_2(\mb{v}_2-\mb{v}_f)|w|
+(\Cauchy_1\mb{n}+\Cauchy_2\mb{n}-\Cauchy\mb{n})}d\bf S\\
=&&\int_\Gamma \paren{\gamma_2\paren{\frac{1}{\phi_2} - 1}w|w|\mb{n}
+(\Cauchy_1\mb{n}+\Cauchy_2\mb{n}-\Cauchy\mb{n})}d\bf S\\
=&&\ 0
\end{eqnarray*}
where we used (\ref{normal-velocity}) in the second and third equalities, (\ref{tangent-velocity})
in the third equality and
(\ref{linear-momentum -gamma}) in the last equality. 
\medskip

\noindent
{\sl \underline{Balance of energy across $\Gamma$.}}
Next, we consider energy conservation. The final form of our
energy relation will lead us to possible forms for the additional
boundary conditions we shall impose on our system.
\begin{eqnarray}\label{encalc}
&&\D{}{t}\paren{\int_{\Omega_t}
\paren{\frac{1}{2}\gamma_1\phi_1\|{\mb{v}_1}\|^2
+\frac{1}{2}\gamma_2\phi_2\|{\mb{v}_2}\|^2}
d\mb{x}+\int_{\Omega^c_t}\frac{1}{2}\gamma_2\|{\mb{v}_f}\|^2d\mb{x}}\\
=&&\int_{\Omega_t}
\paren{\mb{v}_1 (\nabla \cdot \Cauchy_1)+\mb{v}_2(\nabla \cdot \Cauchy_2)}d\mb{x}
+\int_{\Omega^c_t}\mb{v}_f(\nabla \cdot \Cauchy)d\mb{x} \nonumber\\
&&+\int_{\Gamma}\paren{
\frac{1}{2}\gamma_2\phi_2\|{\mb{v}_2}\|^2(\mb{v}_1-\mb{v}_2)\cdot \mb{n}
-\frac{1}{2}\gamma_2\|{\mb{v}_f}\|^2(\mb{v}_1-\mb{v}_f)\cdot \mb{n}}d\bf S \nonumber \\
=&&-\int_{\Omega_t}
\paren{(\nabla \mb{v}_1) \Cauchy_1+(\nabla\mb{v}_2)\Cauchy_2}d\mb{x}
-\int_{\Omega^c_t}(\nabla\mb{v}_f) \Cauchy d\mb{x} \nonumber \\
&&+\int_\Gamma \paren{(\Cauchy_1 \mb{n})\cdot \mb{v}_1+(\Cauchy_2\mb{n})\cdot \mb{v}_2-
(\Cauchy\mb{n})\cdot \mb{v}_f}dS \nonumber \\
&&+\int_\Gamma\paren{
\frac{1}{2}\gamma_2\phi_2\|{\mb{v}_2}\|^2(\mb{v}_1-\mb{v}_2)\cdot \mb{n}
-\frac{1}{2}\gamma_2\|{\mb{v}_f}\|^2(\mb{v}_1-\mb{v_f})\cdot \mb{n}}d\bf S. \nonumber
\end{eqnarray}

Let us evaluate the last two boundary integrals.
Using (\ref{normal-velocity}) and (\ref{tangent-velocity}) we have:
\begin{eqnarray}\label{T12f}
&&(\Cauchy_1 \mb{n})\cdot \mb{v}_1+(\Cauchy_2\mb{n})\cdot \mb{v}_2-
(\Cauchy\mb{n})\cdot \mb{v}_f \\
=&&(\Cauchy_1\mb{n}+\Cauchy_2\mb{n}-\Cauchy\mb{n})\cdot \mb{v}_1
+
\paren{\mb{n}\cdot \paren{\frac{\Cauchy_2}{\phi_2}\mb{n}}-\mb{n}\cdot (\Cauchy\mb{n})}w
+(\Cauchy_2\mb{n}-\Cauchy\mb{n})\cdot \mb{q} \nonumber\\
=&&-\gamma_2w^2\paren{1-\frac{1}{\phi_2}}(\mb{v}_1\cdot\mb{n})
+
\paren{\mb{n}\cdot \paren{\frac{\Cauchy_2}{\phi_2}\mb{n}}-\mb{n}\cdot (\Cauchy\mb{n})}w
-(\Cauchy_1\mb{n})\cdot \mb{q} \nonumber
\end{eqnarray}
where we used (\ref{normal-velocity}) and (\ref{tangent-velocity}) in the first equality and (\ref{Tq})
in the second equality.
On the other hand,
\begin{eqnarray}\label{KE}
&&\frac{1}{2}\gamma_2\phi_2\|{\mb{v}_2}\|^2(\mb{v}_1-\mb{v}_2)\cdot \mb{n}
-\frac{1}{2}\gamma_2\|{\mb{v}_f}\|^2(\mb{v}_1-\mb{v_f})\cdot \mb{n}\\
=&&-\frac{1}{2}\gamma_2\|{\mb{v}_1+\frac{w}{\phi_2}\mb{n}+\mb{q}}\|^2w
+\frac{1}{2}\gamma_2\|{\mb{v}_1+w\mb{n}+\mb{q}}\|^2w \nonumber\\
=&&\gamma_2w^2\paren{1-\frac{1}{\phi_2}}(\mb{v}_1\cdot \mb{n})
-\paren{\frac{1}{2}\gamma_2\paren{\frac{w}{\phi_2}}^2-\frac{1}{2}\gamma_2w^2}w \nonumber
\end{eqnarray}
where we used (\ref{normal-velocity}) in the first equality.
We may go back to (\ref{encalc}) to conclude that:
\begin{eqnarray}\label{encalc2}
&&\D{}{t}\paren{\int_{\Omega_t}
\paren{\frac{1}{2}\gamma_1\phi_1\|{\mb{v}_1}\|^2
+\frac{1}{2}\gamma_2\phi_2\|{\mb{v}_2}\|^2}
d\mb{x}+\int_{\Omega^c_t}\frac{1}{2}\gamma_2\|{\mb{v}_f}\|^2d\mb{x}}\\
=&&-\int_{\Omega_t}
\paren{(\nabla \mb{v}_1) \Cauchy_1+(\nabla\mb{v}_2)\Cauchy_2}d\mb{x}
-\int_{\Omega^c_t}(\nabla\mb{v}_f) \Cauchy d\mb{x} \nonumber \\
&&-
\int_\Gamma
\paren{\paren{\mb{n}\cdot (\Cauchy\mb{n})-\frac{1}{2}\gamma_2w^2}-
\paren{\mb{n}\cdot\paren{\frac{\Cauchy_2}{\phi_2}}\mb{n}-\frac{1}{2}\gamma_2\paren{\frac{w}{\phi_2}}^2}
}w\ d\bf S \nonumber\\
&&-\int_\Gamma (\Cauchy_1\mb{n})\cdot \mb{q}\ d\bf S. \nonumber
\end{eqnarray}
The last two boundary integrals denote the change in energy coming from
the surface $\Gamma$. We would like these terms to be negative.
One way to achieve this would be to let:
\begin{eqnarray}
\eta_\perp w&&=\paren{\mb{n}\cdot (\Cauchy\mb{n})-\frac{1}{2}\gamma_2w^2}-
\paren{\mb{n}\cdot\paren{\frac{\Cauchy_2}{\phi_2}}\mb{n}-\frac{1}{2}\gamma_2\paren{\frac{w}{\phi_2}}^2},
\label{etaw}\\
\eta_\parallel \mb{q}&&=(\Cauchy_1\mb{n})_\parallel\label{etaq},
\end{eqnarray}
where $\eta_\perp$ and $\eta_\parallel$ are positive constants
and $(\Cauchy_1\mb{n})_\parallel$ denotes the component of $\Cauchy_1\mb{n}$ that is
tangential to the membrane.
The above conditions provide the additional three boundary conditions we need
on $\Gamma$.
The boundary condition (\ref{etaw}) depends quadratically on $w$ and
is physically reasonable only if $w$ is sufficiently small. This difficulty
will
not arise if we neglect inertial terms and set $\gamma_i=0, i=1,2$.
For most practical situations, inertial effects can be safely neglected.
If we substitute (\ref{etaw}) and (\ref{etaq}) into (\ref{encalc2}), we have:
\begin{eqnarray*}
&&\D{}{t}\paren{\int_{\Omega_t}
\paren{\frac{1}{2}\gamma_1\phi_1\|{\mb{v}_1}\|^2
+\frac{1}{2}\gamma_2\phi_2\|{\mb{v}_2}\|^2}
d\mb{x}+\int_{\Omega^c_t}\frac{1}{2}\gamma_2\|{\mb{v}_f}\|^2d\mb{x}}\\
=&&-\int_{\Omega_t}
\paren{(\nabla \mb{v}_1) \Cauchy_1+(\nabla\mb{v}_2)\Cauchy_2}d\mb{x}
-\int_{\Omega^c_t}(\nabla\mb{v}_f) \Cauchy d\mb{x}
-\int_\Gamma \paren{\eta_\perp w^2+\eta_\parallel \|{\mb{q}}\|^2}d\bf S.
\end{eqnarray*}

 \subsection{Energy dissipation and constitutive equations}\label{subsec_energydiss}
The free energy density $\Psi$ of the gel consists of the elastic energy  $\mathcal{W}_P(F)$ of the polymer and the Flory--Huggins energy of mixing $\wfh(\phi_1, \phi_2)$ \cite{Fl, Gas}:
\begin{eqnarray}
&&\Psi = \phi_1 \mathcal{W}_P(F)+ \wfh(\phi_1, \phi_2), \quad \textrm {with} \label{freeenergy}\\
&&\wfh(\phi_1, \phi_2)=a\, \phi_1\log\phi_1 + b\,\phi_2\log \phi_2 +c\, \phi_1\phi_2, \label{fhenergy}\\
&&a=\frac{K_BT }{V_mN_1},b=\frac{K_BT}{V_mN_2}, \,\, c=\frac{K_B T}{2V_m}\chi, \nonumber
\end{eqnarray}
where  various parameters above are of the following physical interpretations:
\begin{enumerate}
\item $K_B$ is the Boltzmann constant, and $T$ is the absolute temperature;
\item $V_m$ is the volume occupied by one monomer;
\item $N_1, N_2$ are the numbers of lattice sites occupied by the polymer and the solvent, respectively;
\item $\chi = \chi(\phi_1, \phi_2)$ is the Flory interaction parameter;
\end{enumerate}
In this way, the total energy of the gel is given by
\begin{eqnarray}\label{energy}
\mathcal{E}&&=\int_{\Omega_t} \LC {\gamma_1\over2}\phi_1\|\bv_1\|^2+{\gamma_2\over2}\phi_2\|\bv_2\|^2+\Psi\RC d\bx + \int_{\Omega_t^c} {\gamma_2\over2}\|\bv_f\|^2d\bx.\\
&&:= \mathcal{E}_{P} + \mathcal{E}_{S}. \nonumber
\end{eqnarray}

We now show that smooth solutions to the governing equations satisfy a dissipation inequality, and the dissipation inequality suggests the exact form of the Cauchy stress tensors for the $i$th component, $\Cauchy_i.$  We define $\Cauchy_i$ as the sum of reversible stress $\Cauchy_i^{(r)}$ and viscous stress $\Cauchy_i^{(v)}$ for $i=1,2$: 
\begin{equation}\label{stressdecomp}
	\Cauchy_i = \Cauchy_i^{(r)} + \Cauchy_i^{(v)}.
\end{equation}	
Specifically, the dissipation inequality allows us to obtain the exact form of the reversible stress 
$\Cauchy_i^{(r)}$, the viscous stress $\Cauchy_i^{(v)}$, and the expressions for the friction forces 
$\ff_i$. The derivation of the dissipation inequality uses a similar approach as in~\cite{CCLZ}, and 
we here follow the presentation from~\cite{calderer-chen-micek-mori2010}. 


\begin{theorem}{{\rm (Dissipation Relation)}}\label{thm:dissipation-mechnic}
Suppose that $\{\bv_i, \phi_i\}$ are smooth solutions of  equations (\ref{sat})--(\ref{div}) with boundary conditions (\ref{normal-velocity}), (\ref{tangent-velocity}), (\ref{linear-momentum -gamma}), (\ref{etaw}) and (\ref{etaq}) on $\Gamma$. Let $p$ denote the Lagrange multiplier corresponding to the constraint (\ref{div}).
Assume that the following constitutive equations for the stress tensors components hold: 
\begin{eqnarray}
&&\Cauchy_1^{(r)}= \phi_1\frac{\partial \mathcal{W}_P}{\partial F}F^T-\LB\phi_1\LC\frac{\partial\wfh}{\partial \phi_1}-\frac{\partial\wfh}{\partial\phi_2}\RC-\wfh+p\phi_1\RB\Id,\label{eq:cauchy-relaxation}\\
&&\Cauchy_2^{(r)}= -\phi_2p\Id,\\ 
&&\Cauchy_i^{(v)}=\eta_iD(\bv_i)+\mu_i(\nabla\cdot\bv_i)\Id \label{eq:cauchy-viscous}
\end{eqnarray}
with  $ i=1,2,$ where $\eta_i>0$ and $\mu_i>0$  denote the shear and bulk viscosities of the $i$th component, respectively, and $D(\bv):=\frac{1}{2}(\nabla\bv+\nabla\bv^T)$ is the symmetric part of the velocity gradient.  Suppose that the solvent-polymer friction forces  are given by 
\begin{equation}  
\ff_1=p\nabla\phi_1-\kappa(\bv_1-\bv_2), \quad \ff_2=-\ff_1,  \label{drag-force} 
\end{equation}
with $\kappa>0$ denoting the gel permeability coefficient. 
Then the following dissipation relation holds:
\begin{eqnarray}
\frac{d\mathcal{E}}{dt}=&&-\int_{\Omega_t}\LB\sum_i\LC\eta_i\|D(\bv_i)\|^2 + \mu_i(\nabla\cdot\bv_i)^2\RC + \kappa\|\bv_1-\bv_2\|^2\RB d\bx \nonumber\\
&&-\int_\Gamma \paren{\eta_\perp w^2+\eta_\parallel \|{\mb{q}}\|^2}d\bf S\label{energydissipation}
\end{eqnarray}
where $\bn$ represents the unit outer normal to $\Omega_t$.  
\end{theorem}

\begin{proof}
 Without loss of generality, assume that $\gamma_1=\gamma_2=1$. 
We observe that the boundary of the gel is determined by that of the polymer network and moves with velocity $\bv_1$. We first calculate the time derivative of 
$\mathcal{E}$. 
Using the Reynolds Transport Theorem,  and taking into account that the boundary $\partial \Omega_t$ moves with the network speed $\bv_1$, we have
\begin{eqnarray*}
\frac{d\mathcal{E}_P}{dt}=&& \int_{\Omega_t}\bigg[\frac{\partial}{\partial t}\left(\frac{\phi_1}{2} \|\bv_1\|^2+ \frac{\phi_2}{2}\|\bv_2\|^2\right)+ 
\nabla\cdot\left(\bv_1 \frac{\phi_1}{2}\|\bv_1\|^2\right)   +\nabla\cdot\left(\bv_1 \frac{\phi_2}{2}\|\bv_2\|^2\right)\bigg]\,d\bx\\
&&+\int_{\Omega_t}\bigg[\frac{\partial \Psi}{\partial t}+ \nabla \cdot (\bv_1\phi_1 \mathcal{W}_P) + \nabla\cdot(\bv_1\wfh)\bigg]\,d\bx.
\end{eqnarray*}
Next, expand $\frac{\partial\Psi}{\partial t}$ in terms of $\phi_1, \phi_2, F$ using the chain rule to obtain
\begin{eqnarray*}
\frac{d\mathcal{E}_P}{dt}= &&\int_{\Omega_t}\bigg[\frac{\partial}{\partial t}\left(\frac{\phi_1}{2} \|\bv_1\|^2+ \frac{\phi_2}{2}\|\bv_2\|^2\right)+ 
\nabla\cdot\left(\bv_1 \frac{\phi_1}{2}\|\bv_1\|^2\right)   +\nabla\cdot\left(\bv_1 \frac{\phi_2}{2}\|\bv_1\|^2\right)\bigg]\,d\bx\\
&&+\int_{\Omega_t}\bigg[\left(\mathcal{W}_P+\frac{\partial\wfh}{\partial\phi_1}\right)\phi_{1,t}+\left(\frac{\partial\wfh}{\partial\phi_2}\right)\phi_{2,t}+  \phi_1\frac{\partial \mathcal{W}_P}{\partial F}:F_t\bigg]\,d\bx\nonumber \\
&&+ \int_{\Omega_t}  \bigg[(\nabla\cdot\bv_1)\phi_1 \mathcal{W}_P+ \bv_1\cdot\nabla(\phi_1 \mathcal{W}_P)+ \nabla\cdot(\bv_1\wfh)\bigg]\,d\bx. 
\end{eqnarray*}
  Equation (\ref{mass}) of  balance of mass,
$\frac{\partial\phi_i}{\partial t} =-\nabla\cdot(\phi_i\bv_i)$,
for $i=1,2$ yields 
\begin{eqnarray*}
\frac{d\mathcal{E}_P}{dt}=&& \int_{\Omega_t}\bigg[\frac{\partial}{\partial t}\left(\frac{\phi_1}{2} \|\bv_1\|^2+ \frac{\phi_2}{2}\|\bv_2\|^2\right)+ 
\nabla\cdot\left(\bv_1 \frac{\phi_1}{2}\|\bv_1\|^2\right)   +\nabla\cdot\left(\bv_1 \frac{\phi_2}{2}\|\bv_2\|^2\right)\bigg]\,d\bx\\
&&+\int_{\Omega_t}\bigg[-\left(\mathcal{W}_P+\frac{\partial\wfh}{\partial\phi_1}\right)\nabla\cdot(\phi_1\bv_1)-\left(\frac{\partial\wfh}{\partial\phi_2}\right)\nabla\cdot(\phi_2\bv_2)+  \phi_1\frac{\partial \mathcal{W}_P}{\partial F}: F_t\bigg]\,d\bx\\
&&+ \int_{\Omega_t}  \bigg[(\nabla\cdot\bv_1)\phi_1 \mathcal{W}_P+ \bv_1\cdot\nabla(\phi_1 \mathcal{W}_P)+ \nabla\cdot(\bv_1\wfh)\bigg]d\bx.  %
\end{eqnarray*}
Gathering the terms involving $\wfh$ and using the incompressibility of mixture constraint in equation (\ref{sat}) gives the following:
\begin{eqnarray*}
&& -\frac{\partial\wfh}{\partial\phi_1}\nabla\cdot(\phi_1\bv_1)-\frac{\partial\wfh}{\partial\phi_2}\nabla\cdot(\bv_2\phi_2) +\bv_1\cdot\nabla \wfh+\wfh\nabla\cdot\bv_1\\
 &&=\nabla\cdot\LC\phi_1\bv_1)\big(-\frac{\partial \wfh}{\partial\phi_1}+\frac{\partial \wfh}{\partial\phi_2}\RC+\bv_1\cdot\LC\frac{\partial \wfh}{\partial\phi_1}\nabla\phi_1+ \frac{\partial \wfh}{\partial\phi_2}\nabla\phi_2\RC + \wfh\nabla\cdot\bv_1\\
 &&= \LB\phi_1\LC\frac{\partial \wfh}{\partial\phi_2}-\frac{\partial \wfh}{\partial\phi_1}\RC+\wfh\RB(\nabla\cdot\bv_1)
 \end{eqnarray*}
 Substituting it  into the expression for $\frac{d\mathcal{E}_P}{dt}$ gives
\begin{eqnarray*}
\frac{d\mathcal{E}_P}{dt}=&& \int_{\Omega_t}\bigg[\frac{\partial}{\partial t}\left(\frac{\phi_1}{2} \|\bv_1\|^2+ \frac{\phi_2}{2}\|\bv_2\|^2\right)+ 
\nabla\cdot\left(\bv_1 \frac{\phi_1}{2}\|\bv_1\|^2\right)   +\nabla\cdot\left(\bv_1\frac{\phi_2}{2}\|\bv_2\|^2\right)\bigg]\,d\bx\\
&&+\int_{\Omega_t}\LCB\phi_1\frac{\partial \mathcal{W}_P}{\partial F}: F_t+\LB\phi_1\LC\frac{\partial \wfh}{\partial\phi_2}-\frac{\partial \wfh}{\partial\phi_1}\RC+\wfh\RB(\nabla\cdot\bv_1)
+ \phi_1\bv_1\cdot \nabla \mathcal{W}_P\RCB \,d\bx. 
 \end{eqnarray*}

Next, we apply the chain-rule relation  between material and spatial time derivatives,  
\begin{equation} 
\frac{D F}{dt}= F_t+(\bv_1\cdot \nabla) F= (\nabla\bv_1)F \label{chain-rule},
\end{equation}
to obtain the following identity:
\begin{eqnarray*}
	&&\phi_1\frac{\partial\mathcal{W}_P}{\partial F}: F_t + \phi_1\bv_1\cdot\nabla\mathcal{W}_P \\
	&&=\phi_1\frac{\partial\mathcal{W}_P}{\partial F}: F_t +\phi_1\bv_1\bigg(\frac{\partial\mathcal{W}_P}{\partial F}:\nabla F\bigg)\\
	&&=\phi_1\frac{\partial\mathcal{W}_P}{\partial F}:( F_t + \bv_1\cdot\nabla F)=\phi_1\frac{\partial\mathcal{W}_P}{\partial F}:\nabla\bv_1 F.
\end{eqnarray*}
The expression for $\frac{d\mathcal{E}_P}{dt}$ simplifies to
 \begin{eqnarray*}
\frac{d\mathcal{E}_P}{dt}=&& \int_{\Omega_t}\bigg[\frac{\partial}{\partial t}\left(\frac{\phi_1}{2} \|\bv_1\|^2+ \frac{\phi_2}{2}\|\bv_2\|^2\right)+ 
\nabla\cdot\left(\bv_1 \frac{\phi_1}{2}\|\bv_1\|^2\right)   +\nabla\cdot\left(\bv_1 \frac{\phi_2}{2}\|\bv_2\|^2\right)\bigg]\,d\bx\\
&&+\int_{\Omega_t}  \LCB\phi_1\frac{\partial\mathcal{W}_P}{\partial F}: ((\nabla\bv_1) F)+\
\LB\phi_1\LC\frac{\partial \wfh}{\partial\phi_2}-\frac{\partial \wfh}{\partial\phi_1}\RC+\wfh\RB(\nabla\cdot\bv_1)\RCB\,d\bx\\
 \label{energy-derivative3}
 \end{eqnarray*}

In the final steps, we address the kinetic energy terms.  
Using the equations of balance of mass and linear momentum of for $i=1$, we calculate:
\begin{eqnarray*}
\frac{\partial}{\partial t}\left(\frac{1}{2}\phi_1\|\bv_1\|^2\right)+&&\nabla\cdot\left(\bv_1 \frac{1}{2}\phi_1\|\bv_1\|^2\right)\\
\quad&&=\phi_1\bv_1\cdot{\bv}_{1,t}+\frac{1}{2}\|\bv_1\|^2\bigg(\frac{\partial\phi_1}{\partial t} + \nabla\cdot(\phi_1\bv_1)\bigg)+ \frac{1}{2}\phi_1bv_1\cdot\nabla\big(\|\bv_1\|^2\big)\\ 
&&=\phi_1\bv_1\cdot{\bv}_{1,t} -\frac{1}{2}\|\bv_1\|^2\nabla\cdot(\phi_1\bv_1) + \frac{1}{2}\|\bv_1\|^2\nabla\cdot(\phi_1\bv_1) + \phi_1\bv_1\cdot (\nabla\bv_1)^T\bv_1\\ 
&& =\phi_1\bv_1\cdot \frac{D\bv_1}{Dt}=\bv_1\cdot \LC\nabla\cdot\Cauchy_1+ \ff_1\RC
\end{eqnarray*}
Here, $\frac{D}{Dt}$ denotes the material time derivative.  Likewise,
\begin{eqnarray*}
\frac{\partial}{\partial t}\left(\frac{1}{2}\phi_2\|\bv_2\|^2\right)+&&\nabla\cdot\left(\bv_1 \frac{1}{2}\phi_2\|\bv_2\|^2\right) \\
\quad&&=\phi_2\bv_2\cdot{\bv}_{2,t}+\frac{1}{2}\|\bv_2\|^2\bigg(\frac{\partial\phi_2}{\partial t} + \nabla\cdot(\phi_2\bv_1)\bigg)+ \frac{1}{2}\phi_2\bv_1\cdot\nabla\big(\|\bv_2\|^2\big)\\ 
&&=\phi_2\bv_2\cdot{\bv}_{2,t} -\frac{1}{2}\|\bv_2\|^2\nabla\cdot(\phi_2\bv_2) + \frac{1}{2}\|\bv_2\|^2\nabla\cdot(\phi_2\bv_1) + \phi_2\bv_1\cdot (\nabla\bv_2)^T\bv_2\\ 
&&=\phi_2\bv_2\cdot\bv_{2,t}+\phi_2\bv_2\cdot(\nabla\bv_2)^T\bv_2+\phi_2\bv_1\cdot(\nabla\bv_2)^T\bv_2\\
&&\quad\quad-\phi_2\bv_2\cdot(\nabla\bv_2)^T\bv_2+\frac{1}{2}\|\bv_2\|^2\nabla\cdot(\phi_2(\bv_1-\bv_2))\\
&&= \phi_2\bv_2\cdot \frac{D\bv_2}{Dt}+ \phi_2(\nabla\bv_2)^T\bv_2\cdot(\bv_1-\bv_2)+\frac{1}{2}\|\bv_2\|^2\nabla\cdot(\phi_2(\bv_1-\bv_2)).
\end{eqnarray*}
Integrating the last term by parts gives
\begin{eqnarray*}
\int_{\Omega_t}\bigg[\frac{\partial}{\partial t}&&\left(\frac{1}{2}\phi_2\|\bv_2\|^2\right)+\nabla\cdot\left(\bv_1 \frac{1}{2}\phi_2\|\bv_2\|^2\right)\bigg]\,d\bx\\
\quad &&= \int_{\Omega_t}\bigg[\phi_2\bv_2\cdot \frac{D\bv_2}{Dt}+ \phi_2(\nabla\bv_2)^T\bv_2\cdot(\bv_1-\bv_2)+\frac{1}{2}\|\bv_2\|^2\nabla\cdot(\phi_2(\bv_1-\bv_2))\bigg]\,d\bx\\
&&= \int_{\Omega_t}\phi_2\bv_2\cdot\frac{D\bv_2}{Dt}\,d\bx+\int_{\Gamma_{-}}\frac{1}{2}\phi_2\|\bv_2\|^2(\bv_1-\bv_2)\cdot\bn\,d{\bf S}.
\end{eqnarray*}
Plugging in the Lagrange multiplier we have that, for $i=1,2,$ 
\begin{eqnarray*}
\frac{d\mathcal{E}_P}{dt}=&& \int_{\Omega_t}\sum_i (\bv_i\cdot (\nabla\cdot\Cauchy_i) +{\ff}_i\cdot\bv_i + p\nabla\cdot(\phi_i\bv_i))\,d\bx\\
&&+\int_{\Omega_t} \LCB\phi_1\frac{\partial \mathcal{W}_P}{\partial F}: (\nabla\bv_1)F
+\LB\phi_1\LC\frac{\partial \wfh}{\partial\phi_2}-\frac{\partial \wfh}{\partial\phi_1}\RC+\wfh\RB(\nabla\cdot\bv_1)]\RCB\,d\bx\\
&&+\frac{1}{2}\int_{\partial\Omega_t}\phi_2\|\bv_2\|^2(\bv_1-\bv_2)\cdot\bn\,d{\bf S}.  
  \end{eqnarray*}
Integrating by parts on the term $\nabla\cdot\Cauchy_i$ for $i=1,2$ gives,
\begin{eqnarray}
\frac{d\mathcal{E}_P}{dt}&&= \int_{\Omega_t} \sum_i(-\nabla\bv_i \cdot\Cauchy_i -p\phi_i(\nabla\cdot\bv_i) +{\ff}_i\cdot\bv_i - p\nabla\phi_i\cdot\bv_i)\,d\bx \nonumber\\
&&+\int_{\Omega_t} \LCB\phi_1\frac{\partial \mathcal{W}_P}{\partial F}: (\nabla\bv_1) F+
 \LB\phi_1\LC\frac{\partial \wfh}{\partial\phi_2}-\frac{\partial \wfh}{\partial\phi_1}\RC+\wfh\RB(\nabla\cdot\bv_1)\RCB\,d\bx \nonumber\\
&&+\int_{\partial\Omega_t} \LB\Cauchy_i\bv_i +\frac{1}{2}\phi_2\|\bv_2\|^2(\bv_1-\bv_2)\RB\cdot\bn\,d{\bf S}.\label{gel-energy-derivative}  
\end{eqnarray}

The time derivative $\frac{d\mathcal{E}_S}{dt}$ is treated in a similar way giving
\begin{equation}
 \frac{d{\mathcal E}_S}{dt}= -\int_{\Omega^c_t} \nabla\bv_f\cdot\mathcal T\,d\bx-\int_{\Gamma}\LC\mathcal T\bv_f+ \frac{\|\bv_f\|^2}{2}(\bv_1-\bv_f)\RC\cdot \bn\,dS.\label{solvent-energy-derivative}
\end{equation}

Therefore the conclusion follows by adding up (\ref{gel-energy-derivative}) and (\ref{solvent-energy-derivative}), and using $\Cauchy_i$ as in (\ref{eq:cauchy-relaxation}) and (\ref{eq:cauchy-viscous}), $\ff_i$  as given by (\ref{drag-force}), and all the boundary conditions. 

\end{proof}

We further  assume that the polymer is an isotropic elastic material, that is,  
\begin{equation}
\mathcal{W}_P(F)=\mathcal{{W}}_P(I_1, I_2, I_3),  \quad I_1= \tr{C}, \,\, I_2= \frac{1}{2}\LB \tr^2(C)- (\tr \,C)^2\RB, \,\, I_3=\det\, C, \end{equation}
where $C=F^TF$. Here we take $\mathcal{W}_P(F)$ to be of the following form
\begin{equation}\label{elasticenergy}
\mathcal{W}_P(F)=(I_1^s - c) + \alpha_0(I_3^{-{r\over2}} - 1) + \beta_0I_3^{{1\over2}} + \beta_1 I_3^{{q\over2}}
\end{equation}
where $c, \alpha_0, \beta_0, \beta_1>0$, $r\geq1$ and $s,q>1$. 

As for the Flory-Huggins mixture energy, we adopt the interaction equation from Horkay et al. \cite{Horkay} as follows.
\begin{equation}\label{interaction}
\chi(\phi_1, \phi_2) = \chi_0 + \chi_1 \phi_1 + \chi_2 \phi_1^2.
\end{equation}

Further computation gives the following results for the reversible stress tensors
\begin{eqnarray}
&&\Cauchy_1^{(r)}=2sI_1^{s-1}FF^T + \LB -\alpha_0r(\det F)^{-r} + \beta_0\det F + \beta_1 q(\det F)^q \RB \Id \nonumber\\
&&\quad \quad \ \ - \LCB \phi_1 \LB \LC {K_BT\over 2V_m}\chi_0\phi_2 + {K_BT\over N_1V_m}\log\phi_1 + {K_BT\over N_1V_m} + 2\chi_1\phi_1\phi_2 + 3\chi_2\phi^2_1\phi_2 \RC \right.\right. \nonumber\\
&&\quad \quad \quad \quad \ \ \left. - \LC {K_BT\over 2V_m}\chi_0\phi_1 + {K_BT\over N_2V_m} \log\phi_2 + {K_BT\over N_2V_m} + \chi_1\phi^2_1 + \chi_2\phi_1^3 \RC\RB \label{stress1}\\
&&\quad \quad \quad \ \ - \LC  {K_BT\over 2V_m} \chi_0\phi_1\phi_2 +  {K_BT\over N_1V_m}\phi_1\log\phi_1 +  {K_BT\over N_2V_m}\phi_2\log\phi_2 \RC \nonumber\\
&&\quad \quad \quad \ \  + \chi_1\phi_1^2\phi_2 + \chi_2\phi^2_1\phi_2 + p\phi_1 \Big\} \Id, \nonumber\\
&&\Cauchy_2^{(r)}= -p\phi_2 \Id. \label{stress2}
\end{eqnarray}


\subsection{A new field of unknowns}\label{subsec_1d}

Now we will formulate the effective governing equations in terms of the center of mass velocity $\textbf{V}=\phi_1\textbf{v}_1+\phi_2\textbf{v}_2$ and the diffusion velocity $\textbf{U} = \textbf{v}_1-\textbf{v}_2$. Hence the new field of unknowns is
\begin{eqnarray*}
\{ \textbf{V}, \textbf{U}, \phi_1, p, F \}.
\end{eqnarray*}
The total stress of the system is
\begin{eqnarray*}
\mathcal{T}=\mathcal{T}_1^{(r)}+\mathcal{T}_2^{(r)}-\phi_1\phi_2 \textbf{U}\otimes\textbf{U},
\end{eqnarray*}

In this way, in Eulerian coordinates, the governing equations (\ref{sat})--(\ref{div}) become
\begin{eqnarray}
&&\frac{\partial\phi_1}{\partial t}+\LB (\textbf{V}+(1-\phi_1)\textbf{U})\cdot\nabla \RB \phi_1+\phi_1\nabla\cdot \LB \textbf{V}+(1-\phi_1)\textbf{U} \RB =0, \label{eqn_mass}\\
&&\frac{\partial\textbf{V}}{\partial
t}+(\textbf{V}\cdot\nabla)\textbf{V}=\nabla\cdot\mathcal{T}, \label{eqn_momentum}\\
&&\frac{\partial\textbf{U}}{\partial
t}+(1-2\phi_1)(\nabla\textbf{U})\textbf{U}-(\textbf{U}\otimes\textbf{U})\nabla\phi_1
+(\nabla\textbf{V})\textbf{U}+(\nabla\textbf{U})\textbf{V} \nonumber\\
&&\quad
=\frac{1}{\phi_1}\nabla\cdot\mathcal{T}_1-\frac{1}{1-\phi_1}\nabla\cdot\mathcal{T}_2-
\frac{\beta}{\phi_1(1-\phi_1)}\textbf{U}+\frac{\lambda\nabla\phi_1}{\phi_1(1-\phi_1)}, \label{eqn_micro}\\
&&F_t+\LB \textbf{V}+(1-\phi_1)\textbf{U} \RB \cdot\nabla F=\nabla\LB \textbf{V}+(1-\phi_1)\textbf{U} \RB F, \label{eqn_chain}\\
&&\nabla\cdot\textbf{V}=0. \label{eqn_incomp}
\end{eqnarray}
The first and second equations give the balance of mass and linear momentum for the mixture. The third equation can be interpreted as giving the evolution of the microstructure of the gel. Equation (\ref{eqn_chain}) is a version of the chain rule relating time derivatives of $F$ with velocity gradients. This equation is required in mixed solid-fluid systems \cite{LLZ}. The last equation looks like an incompressibility condition of the mixture and is from the balance of mass and the constraint $\phi_1+\phi_2=1$.

In application to modeling gliding behavior of bacteria by polysaccharide swelling, the problem can be thought of as in one dimensional \cite{Ho}. We consider the gel occupying a strip domain $\{(x,y,z):\ -L\leq x\leq L\}$ for some $L>0$. Now the fields of problem become
\begin{equation}\label{eqn_newfields}
\textbf{V}=(v(x,t),0,0),\quad \textbf{U}=(u(x,t),0,0),\quad \phi_1=\phi(x,t),\quad p=p(x,y,z,t).
\end{equation}
The deformation gradient is 
\begin{equation}\label{eqn_1dF}
F=\hbox{diag}(\hbox{det}F(x,t), 1,1).
\end{equation}
From (\ref{mass-lagrangian}), $\phi(x,t)\det F(x,t)=\phi_I$. We can also write down the stress tensors (\ref{stress1}) and (\ref{stress2}) in the one-dimension formulation
\begin{eqnarray}
\Cauchy_1&&= \phi  \left\{2 s\left(\phi_I^2 \phi^{-2}+2\right)^{s-1}\hbox{diag}\left(\phi_I^2 \phi^{-2},1,1\right)-\alpha _0 \phi_I ^{-r} r \phi {}^r\Id+\beta _0 \frac{\phi_I }{\phi } \Id+\beta_1  q \phi_I ^q\phi ^{-q}\Id\right\}\nonumber\\
&&\quad \quad \ \ - \LCB \phi \LB \LC {K_BT\over 2V_m}\chi_0(1 - \phi) + {K_BT\over N_1V_m}\log\phi + {K_BT\over N_1V_m} + 2\chi_1\phi(1 - \phi) + 3\chi_2\phi^2(1 - \phi) \RC \right.\right. \nonumber\\
&&\quad \quad \quad \quad \ \ \left. - \LC {K_BT\over 2V_m}\chi_0\phi + {K_BT\over N_2V_m} \log(1-\phi) + {K_BT\over N_2V_m} + \chi_1\phi^2 + \chi_2\phi^3 \RC\RB \label{1dstress1}\\
&&\quad \quad \quad \ \ - \LC  {K_BT\over 2V_m} \chi_0\phi(1 - \phi) +  {K_BT\over N_1V_m}\phi\log\phi +  {K_BT\over N_2V_m}(1 - \phi)\log(1 - \phi) \RC \nonumber\\
&&\quad \quad \quad \ \  + \chi_1\phi^2(1 - \phi) + \chi_2\phi^2(1 - \phi) + p\phi \Big\} \Id, \nonumber\\
\Cauchy_2&&= -p(1-\phi) \Id. \label{1dstress2}
\end{eqnarray}

The second and third equations of (\ref{eqn_momentum}) indicate that $p=p(x,t)$ is independent of $y$ and $z$. Moreover, equation (\ref{eqn_incomp}) together with the first component of equation (\ref{eqn_newfields}) gives $v=v(t)$. Prescribing the initial condition $v(0)=0$ leads to $v(t)=0$ for all $t>0$, provided that $\nabla\cdot\mathcal{T}=0$ holds. The latter determines $p$ in terms of $\phi$ and $u$, up to a constant. In this way we arrive at the following system for $\phi$ and $u$
\begin{equation}\label{freebdry}
\left\{\begin{array}{ll}
\phi_t+\LB \phi(1-\phi)u\RB_x=0,\\\\
\displaystyle u_t+\LB\frac12 u^2(1-2\phi)-G(\phi)\RB_x=-\frac{\beta u}{\phi(1-\phi)},
\end{array}\right.
\end{equation}
where
\begin{eqnarray}\label{eqn_G}
G(\phi )
&=&\frac{K_BT \log(1-\phi )}{N_2  V_m }-\frac{K_BT \log(\phi )}{N_1  V_m }+(q-1) \beta_1 \phi_I ^q  \phi ^{-q}-(1+r) \alpha _0 \phi_I ^{-r} \phi ^r \nonumber\\
&&+\left(2+\frac{\phi_I ^2}{\phi ^2}\right)^s \left(\frac{2 s \phi_I ^2}{\phi_I ^2+2 \phi ^2}-1\right) +\frac{K_BT\chi_0 \phi}{V_m } - 2\chi_1\phi + 3 (\chi_1 - \chi_2) \phi^2 + 4 \chi_2 \phi^3.
\end{eqnarray}

\subsection{Boundary conditions}\label{subsec_BC}
We assume initially, the polymer occupies the domain $\Omega_0=\{(x,y,z):\ -L<x<L\}$, and the solvent is in $|x|>L$. At a later time $t>0$, the gel occupies the region $\Omega_t=\{(x,y,z):\ -S_1(t)<x<S_2(t)\}$, where $x=S_{1,2}(t)$ are the positions of the interface between the gel and the pure solvent. Therefore inside the polymer region $-S_1(t)<x<S_2(t)$ equations (\ref{freebdry}) hold for $\phi$ and $u$.

The general boundary conditions are described in Section \ref{subsec_bdrycond}. Here we further simply the problem by ignoring the inertial effects from the system and hence the stress balance (\ref{linear-momentum -gamma}) now becomes
\begin{equation}\label{eqn_stressbdry}
(\mathcal{T}_1+\mathcal{T}_2)\textbf{n}=\mathcal{T}\textbf{n}, \quad \hbox{on}\  \Gamma.
\end{equation}


The other boundary conditions we impose here for our problem characterize the degree of permeability of the interface (see also, for instance, \cite{Doi, YD1, YD2}). Again with no inertial effects, (\ref{etaw}) becomes
\begin{equation}\label{perm}
\eta_\perp w=\mb{n}\cdot (\Cauchy\mb{n}) - \mb{n}\cdot\paren{\frac{\Cauchy_2}{\phi_2}}\mb{n}.
\end{equation}
We now classify the boundary permeability in the following way
\begin{enumerate}
\item[(1)] The interface is fully permeable if $\eta_\perp=0$. Thus
\begin{equation}\label{eqn_fullyper}
\mb{n}\cdot (\Cauchy\mb{n}) - \mb{n}\cdot\paren{\frac{\Cauchy_2}{\phi_2}}\mb{n}=0, \quad \hbox{on}\  \Gamma.
\end{equation}
\item[(2)] The interface is impermeable if $\eta_\perp=\infty$. In this case $w=0$, that is
\begin{equation}\label{eqn_imper}
\textbf{v}_1\cdot\mb{n}=\textbf{v}_2\cdot\mb{n}, \quad \hbox{on}\ \Gamma.
\end{equation}
\item[(3)] The interface is semipermeable for $\eta_\perp\in(0, \infty)$.
\end{enumerate}

In this paper we will consider the fully permeable interface. Hence from (\ref{eqn_stressbdry}) and (\ref{eqn_fullyper}) we know that
\begin{eqnarray*}
\mb{n}\cdot\LC \Cauchy_1-{\phi_1\over\phi_2}\Cauchy_2 \RC\mb{n}=0.
\end{eqnarray*}

In one dimension, plugging in (\ref{1dstress1}), (\ref{1dstress2}) and the above we obtain
\begin{eqnarray}\label{eqn_saturation}
0\ && = \phi  \left\{2 s\left(\phi_I^2 \phi^{-2}+2\right)^{s-1}\phi_I^2 \phi^{-2} -\alpha _0 \phi_I ^{-r} r \phi {}^r+\beta _0 \frac{\phi_I }{\phi } +\beta_1  q \phi_I ^q\phi ^{-q} \right\}\nonumber\\
&&\quad \quad \ \ - \LCB \phi \LB \LC {K_BT\over 2V_m}\chi_0(1 - \phi) + {K_BT\over N_1V_m}\log\phi + {K_BT\over N_1V_m} + 2\chi_1\phi(1 - \phi) + 3\chi_2\phi^2(1 - \phi) \RC \right.\right. \nonumber\\
&&\quad \quad \quad \quad \ \ \left. - \LC {K_BT\over 2V_m}\chi_0\phi + {K_BT\over N_2V_m} \log(1-\phi) + {K_BT\over N_2V_m} + \chi_1\phi^2 + \chi_2\phi^3 \RC\RB \\
&&\quad \quad \quad \ \ - \LC  {K_BT\over 2V_m} \chi_0\phi(1 - \phi) +  {K_BT\over N_1V_m}\phi\log\phi +  {K_BT\over N_2V_m}(1 - \phi)\log(1 - \phi) \RC \nonumber\\
&&\quad \quad \quad \ \  + \chi_1\phi^2(1 - \phi) + \chi_2\phi^2(1 - \phi) \Big\}. \quad \hbox{on }\Gamma. \nonumber
\end{eqnarray}
Therefore we can determine the saturation value $\phi=\phi^*$ at the interface. In other words,  
the locations of the interface $x=-S_1(t)$ and $x=S_2(t)$ are determined by the saturation value $\phi^*$, i.e. $\phi(t, -S_1(t)) = \phi(t, S_2(t))=\phi^*$. 

The kinematic boundary condition asserts that the interface moves with the speed of the polymer, which means $-S'_1(t)=[1-\phi(t, -S_1(t))]u(t, -S_1(t))$ and $S_2'(t)=(1-\phi(t, S_2(t)))u(S_2(t),t)$. Therefore we have obtained the initial and boundary conditions as follows
\begin{equation}\label{IBcond}
\left\{\begin{array}{l}
\phi(x,t)=\phi^*, \hbox{ at } x=-S_1(t), S_2(t)\\
S_1(0)=L, \quad S_1'(t)=-[1-\phi(t, -S_1(t))]u(t, -S_1(t))\\
S_2(0)=L, \quad S_2'(t)=[1-\phi(t, S_2(t))]u(t, S_2(t))\\
\phi(x,0)=\phi^0,\ \ u(x,0)=u^0, \hbox{ for } -L<x<L.
\end{array}\right.
\end{equation}

\section{The transformed problem}\label{sec_fix}

In this section we aim to setup a fixed-boundary problem associated to (\ref{freebdry}) and (\ref{IBcond}) and establish the local-wellposedness of strong solutions.

To transform the free boundary to a fixed boundary, we perform the following change of coordinates
\begin{eqnarray}\label{changecoord}
y=\int_{-S_1(t)}^x \phi(z,t)\ dz, \ \ \tau=t.
\end{eqnarray}
Because $\int^{S_2(t)}_{-S_1(t)}\phi\ dz$ gives the total mass of the polymer, we may normalize that to be 1. In this way the free domain $(-S_1(t), S_2(t))$ becomes the fixed domain $(0,1)$. Therefore the free boundary problem now turns into
\begin{equation}\label{eqn:fixed1}
\left\{\begin{array}{l}
\phi_{\tau}+\phi^2(1-\phi)u_y-\phi^2\phi_yu=0,\\
u_{\tau}-\phi^2uu_y-u^2\phi\phi_y-G'(\phi)\phi\phi_y=\frac{-\beta u}{\phi(1-\phi)},\\
\phi(y,\tau)=\phi^*, \hbox{ for } y=0,1\\
\phi(y,0)=\phi^0,\ \ u(y,0)=u^0, \hbox{ for } 0<y<1.
\end{array}\right.
\end{equation}

To further rewrite the system, we let $\psi=1/\phi$ and $f(s)$ satisfy that $f'(s)=sG'(s)$. Then let $F(s)=f(1/s)$. In this way the above system becomes the following initial-boundary value problem
\begin{equation}\label{eqn:fixed2}
\left\{\begin{array}{l}
\left(\begin{array}{c}
                        \psi \\ u
                        \end{array}\right)_{\tau}+\left(\begin{array}{c}
                                                               \displaystyle    -\left( 1-\frac{1}{\psi} \right)u \\ \displaystyle -\frac{u^2}{2\psi^2}-F(\psi)
                                                                  \end{array}\right)_y=
\left(\begin{array}{c}
        0 \\ \displaystyle \frac{-\beta u\psi}{(1-\frac{1}{\psi})}
        \end{array}\right),\hbox{ in } \ (0,1)\times(0,T)\\\\
\psi=\psi*, \hbox{ at } y=0,1, \\\\
\left(\begin{array}{c}
                        \psi \\ u
                        \end{array}\right)\Big|_{\tau=0}=\left(\begin{array}{c}
                        \psi^0 \\ u^0
                        \end{array}\right), \hbox{ for } 0<y<1,
\end{array}\right.
\end{equation}
where $\psi^*=1/\phi^*$ and $\psi^0=1/\phi^0$.
The gradient matrix is
\begin{eqnarray*}
A(\psi,u)=\left(\begin{array}{cc}
                  \displaystyle      -\frac{u}{\psi^2} & \displaystyle \frac{1-\psi}{\psi} \\\\
                  \displaystyle     \frac{u^2+G'(1/\psi)}{\psi^3} & \displaystyle  -\frac{u}{\psi^2}
                        \end{array}\right)
\end{eqnarray*}
with eigenvalues
\begin{eqnarray}\label{evalues}
\lambda_{1,2}(\psi, u)=\frac{-u\mp
\sqrt{\Big[u^2+G'(1/\psi)\Big](1-\psi)}}{\psi^2},
\end{eqnarray}
and the corresponding left and right eigenvectors are
\begin{eqnarray}\label{leftevectors}
L_{1,2}(\psi,
u)=\Big(\mp\frac{1}{\psi}\sqrt{\frac{u^2+G'(1/\psi)}{1-\psi}},
1\Big) 
\end{eqnarray}
and 
\begin{eqnarray}\label{rightevectors}
R_{1,2}(\psi, u)=\left(\begin{array}{c}
                        \displaystyle\mp\psi\sqrt{\frac{1-\psi}{u^2+G'(1/\psi)}} \\\\
                        1
                        \end{array}\right).
\end{eqnarray}

In the range of physical parameters corresponding to semi-dry polymer we have $0<\phi<1$, hence $1<\psi$. Thus the system (\ref{eqn:fixed2}) is hyperbolic if
\begin{equation}\label{eqn:hypcond1}
u^2+G'(1/\psi)<0.
\end{equation}
Hence $G'(\phi)<0$ will be needed to guarantee hyperbolicity of the governing system, and therefore is a requirement for the propagation of the swelling front towards the solvent region. It turns out that this condition is satisfied for polymer data (see Fig.1). However in the case of polysaccharide data, there may be multiple quantities $\phi_c$ such that $G'(\phi_c)=0$ (see Fig.2).  
This may be interpreted in terms of the onset of deswelling, observed in bacteria motility phenomenon \cite{Ho}; it may also be associated with volume phase transitions observed in systems with a small elastic shear modulus \cite{LT}.

\begin{figure}[h]
\begin{minipage}{.4\textwidth}
\begin{center}
\includegraphics[scale=.33]{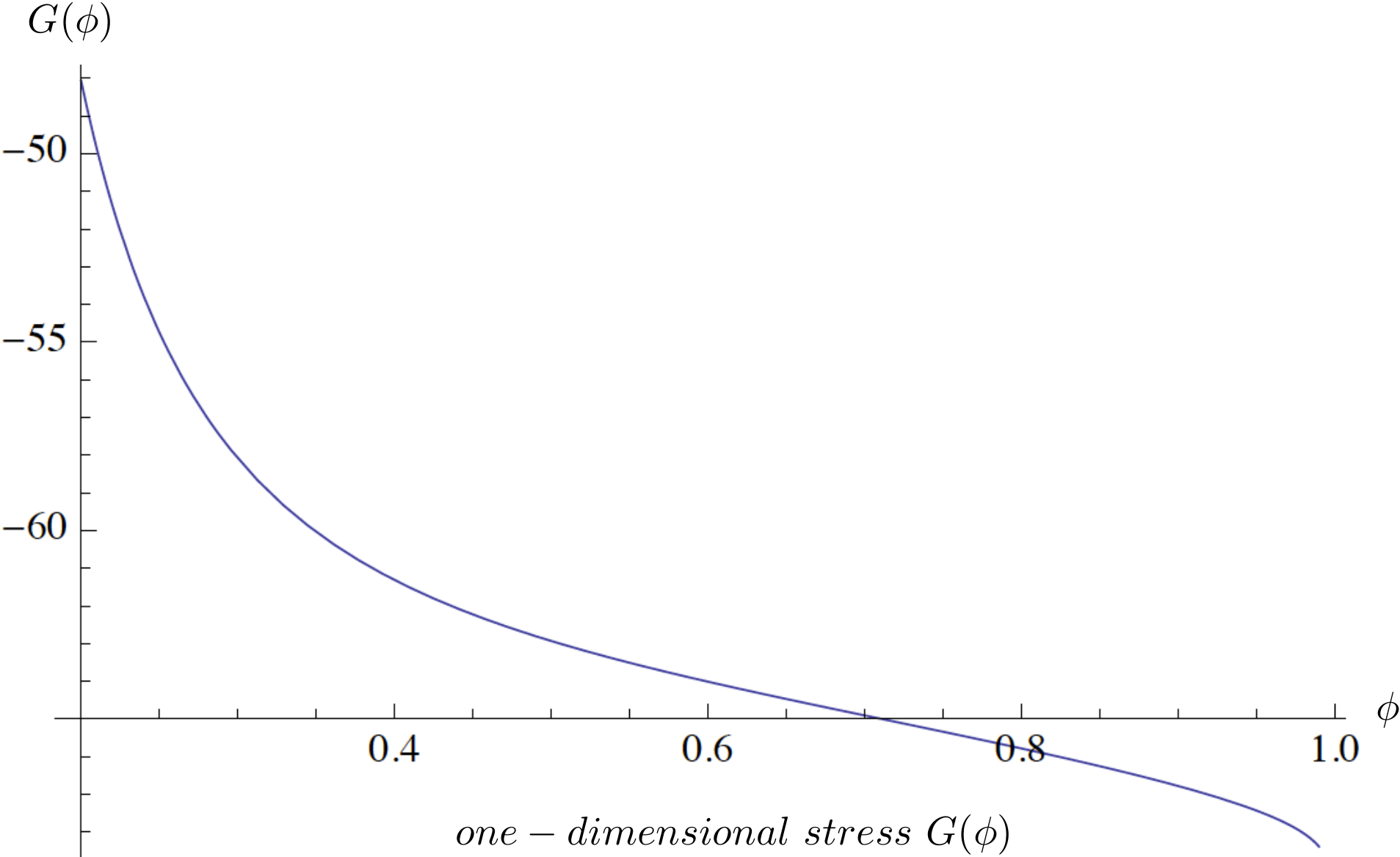}
\end{center}
\end{minipage}\hspace{1.2cm}
\begin{minipage}{.4\textwidth}
\begin{center}
\includegraphics[scale=.28]{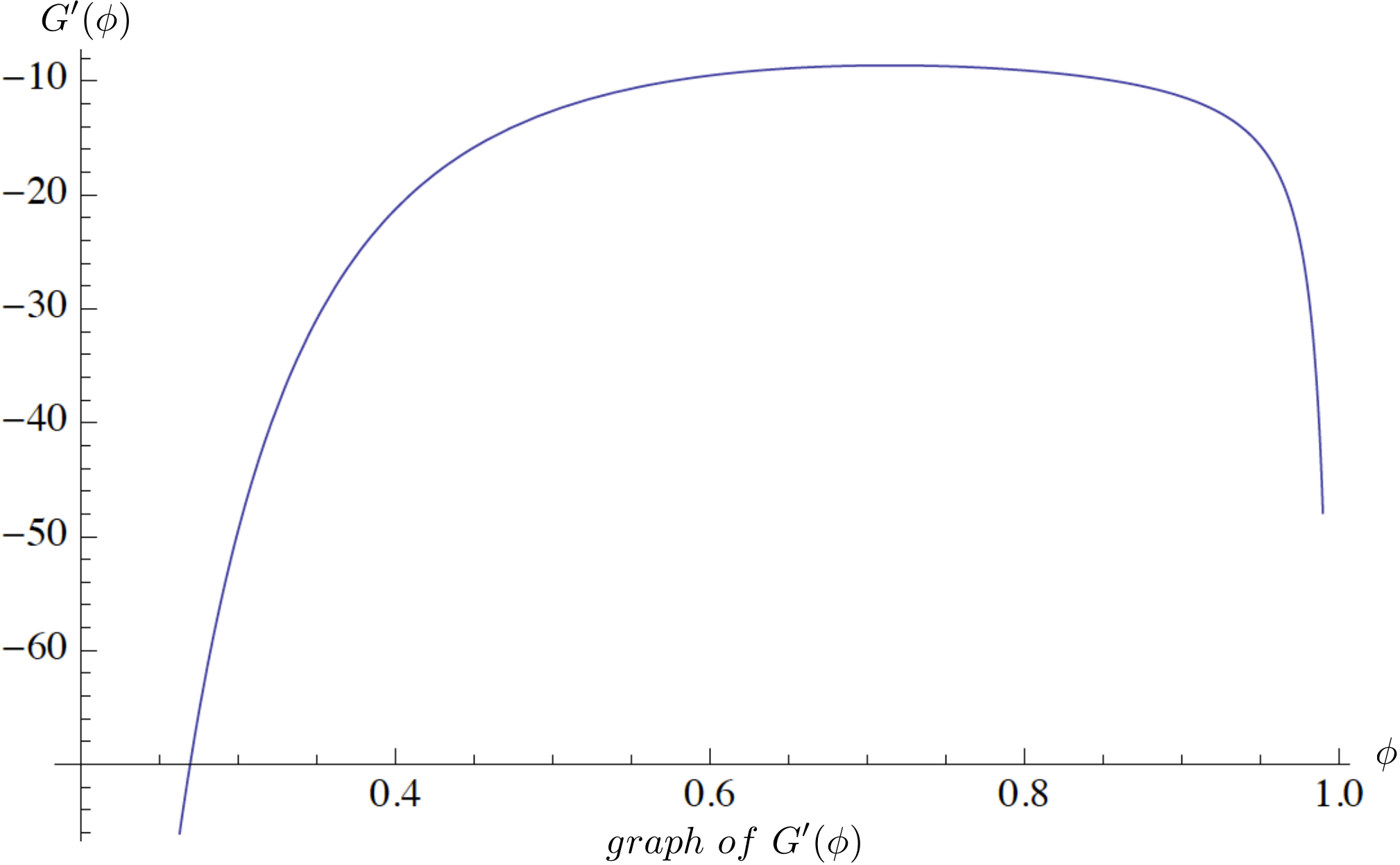}
\end{center}
\end{minipage}
\caption{$G$ and $G'$ for polymer data}
\end{figure}

\begin{figure}[h]
\begin{minipage}{.4\textwidth}
\begin{center}
\includegraphics[scale=.34]{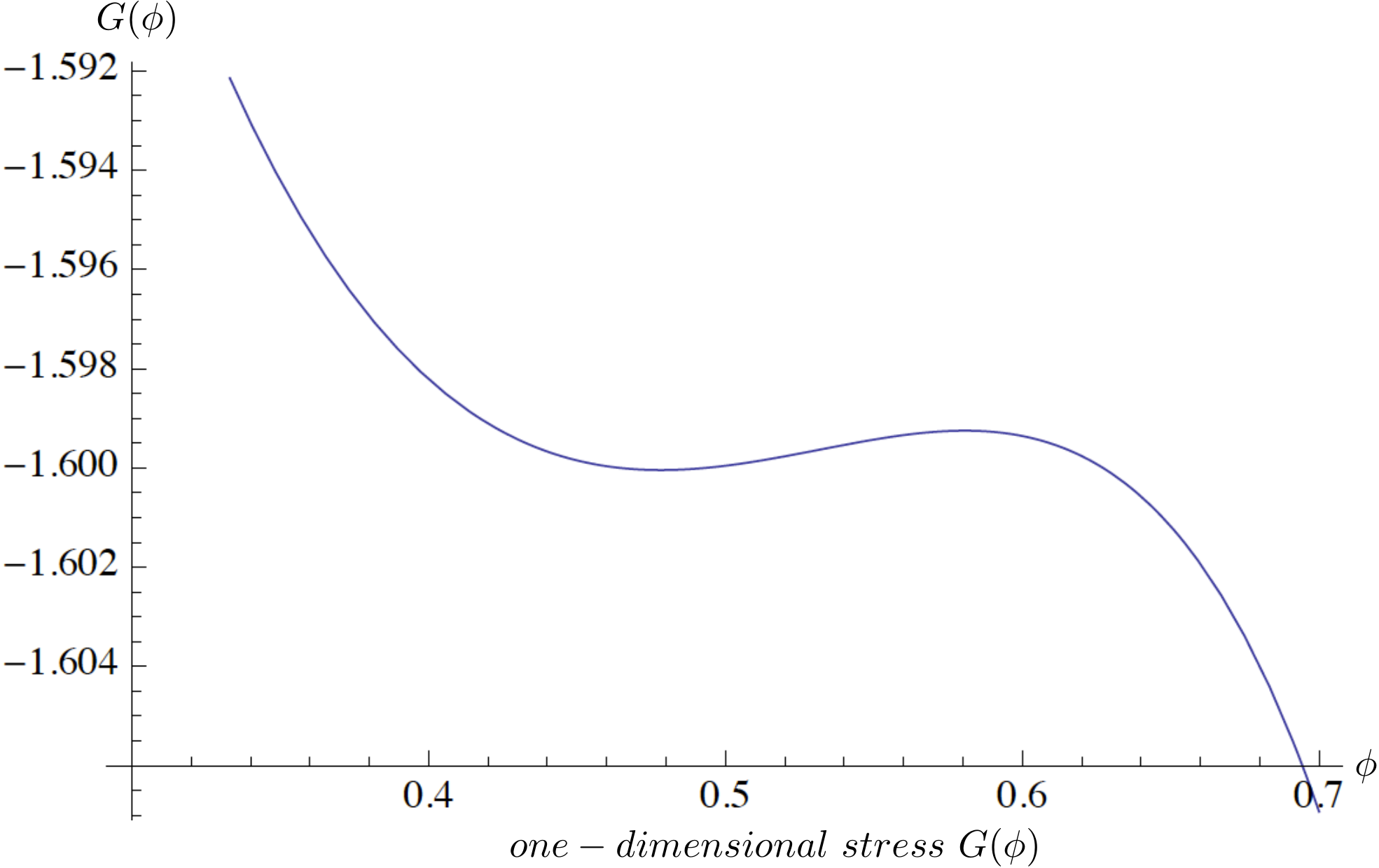}
\end{center}
\end{minipage}\hspace{1.2cm}
\begin{minipage}{.4\textwidth}
\begin{center}
\includegraphics[scale=.28]{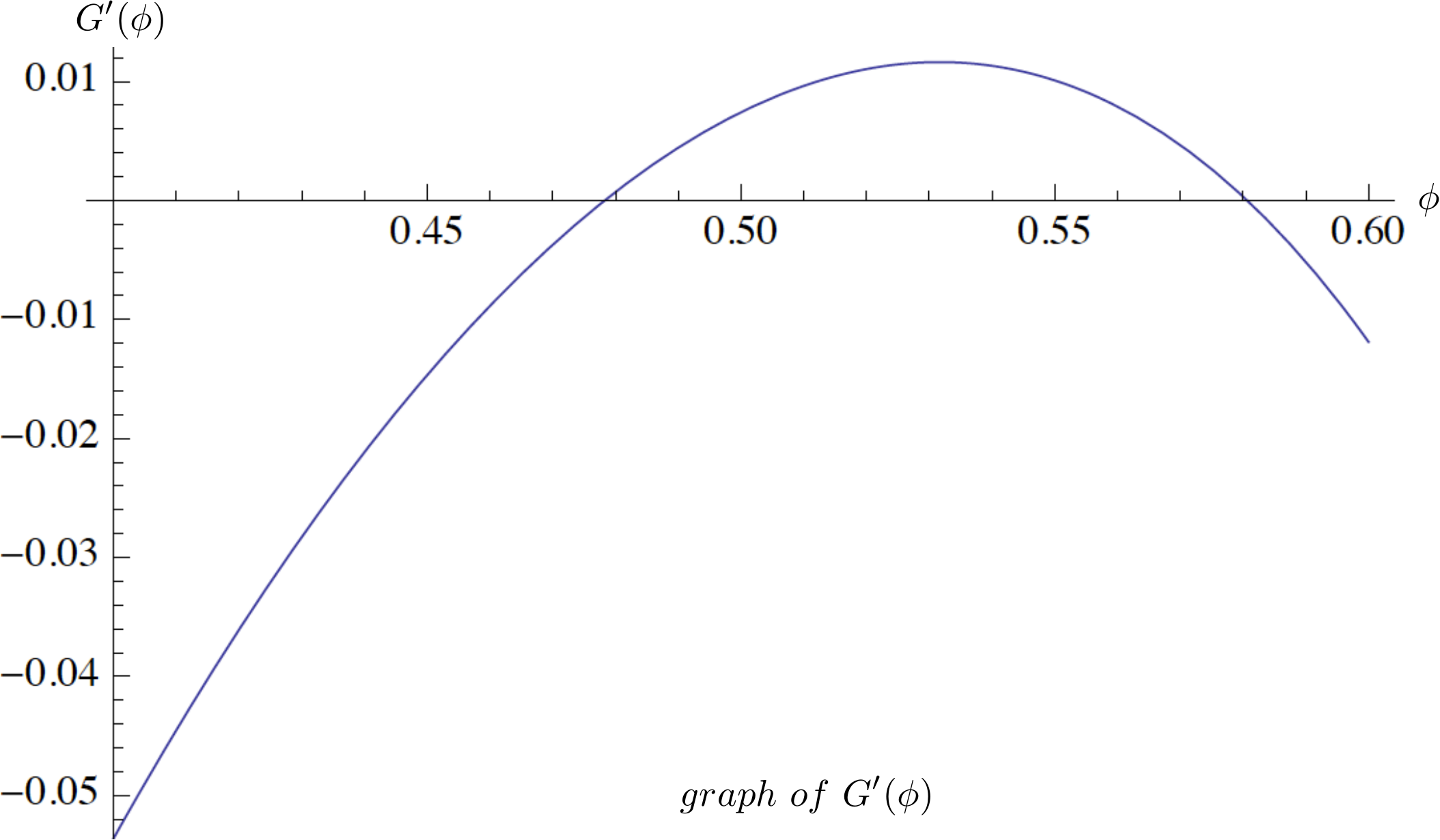}
\end{center}
\end{minipage}
\caption{$G$ and $G'$ for polysaccharide data}
\end{figure}

{\small\begin{center}
\begin{tabular}{|c|c|c|c|c|c|c|c|c|c|c|c|}\hline\rule[-2mm]{0cm}{6mm}
parameter & \hspace{.05cm} $N_1$ \hspace{.05cm} & \hspace{.05cm} $N_2$ \hspace{.05cm} & \hspace{.05cm} $q$ \hspace{.05cm} & \hspace{.05cm} $s$ \hspace{.05cm} & \hspace{.05cm} $r$ \hspace{.05cm} & \hspace{.05cm} $\alpha$ \hspace{.05cm} & \hspace{.05cm} $\beta_1$ \hspace{.05cm} & \hspace{.05cm} $\phi_I$ \hspace{.05cm} & \hspace{.05cm} $\chi_0$ \hspace{.05cm} & \hspace{.05cm} $\chi_1$ \hspace{.05cm} & \hspace{.05cm} $\chi_2$ \hspace{.05cm} \\ 
\hline\rule[-2mm]{0cm}{6mm} polymer & \hspace{.05cm} $1000$ \hspace{.05cm} & \hspace{.05cm} $1$ \hspace{.05cm} & \hspace{.05cm} $N_1$ \hspace{.05cm} & \hspace{.05cm} $6$ \hspace{.05cm} & \hspace{.05cm} $1.25$ \hspace{.05cm} & \hspace{.05cm} $0.001$ \hspace{.05cm} & \hspace{.05cm} $20$ \hspace{.05cm} & \hspace{.05cm} $0.05$ \hspace{.05cm} & \hspace{.05cm} $0.467$ \hspace{.05cm} & \hspace{.05cm} $0.593$ \hspace{.05cm} & \hspace{.05cm} $-0.42$ \hspace{.05cm} \\ 
\hline\rule[-2mm]{0cm}{6mm}
polysaccharide & \hspace{.05cm} $1000$ \hspace{.05cm} & \hspace{.05cm} $1$ \hspace{.05cm} & \hspace{.05cm} $2$ \hspace{.05cm} & \hspace{.05cm} $0.6$ \hspace{.05cm} & \hspace{.05cm} $1.25$ \hspace{.05cm} & \hspace{.05cm} $0.001$ \hspace{.05cm} & \hspace{.05cm} $0.002$ \hspace{.05cm} & \hspace{.05cm} $0.05$ \hspace{.05cm} & \hspace{.05cm} $0.446$ \hspace{.05cm} & \hspace{.05cm} $0.106$ \hspace{.05cm} & \hspace{.05cm} $-0.02$ \hspace{.05cm} \\ \hline
\end{tabular}\\
{\small parameter values}
\end{center}}

\medskip

In concern with the local-wellposedness of system (\ref{eqn:fixed2}), we need to check the
following conditions (see \cite{BD}):

\begin{itemize}
\item[C1.] Non-characteristic condition. The matrix $A(\psi, u)$ is non-singular for $(\psi, u)$ in a certain proper domain $\mathcal{M}$.


\item[C2.] Normality. The boundary matrix $B$ is of constant, maximal rank and
\begin{eqnarray*}
\mathbb{R}^2&&=\hbox{ker}B\oplus E^s(A(\psi, u))\ \ \hbox{ at } \ y=1\\
&&=\hbox{ker}B\oplus E^u(A(\psi, u))\ \ \hbox{ at }\ y=0, 
\end{eqnarray*}
where $E^s(A(\psi, u))$ is the stable subspace of $A(\psi, u)$ and
$E^u(A(\psi, u))$ is the unstable subspace of $A(\psi, u)$.

Note that in our problem, the boundary matrix is 
\begin{eqnarray*}
B=\left(\begin{array}{cc}
        1 & 0\\
        0 & 0
        \end{array}
\right).
\end{eqnarray*}

\item[C3.] Uniform Kreiss-Lopatinski\u\i \ (UKL) condition. for all $(\psi, u)\in \mathcal{M}$ 
there exists $C>0$ so that
\begin{equation}\label{eqn:UKL}
\|V\|\leq C\|BV\|
\end{equation}
for all $V$ in the unstable subspace of $A^{-1}(\psi,u)$ at $y=0$ and for all $V$ in the
stable subspace of  $A^{-1}(\psi,u)$ at $y=1$.

\end{itemize}

Consider now our problem (\ref{eqn:fixed2}), the first two conditions (C1) and (C2) are equivalent to
\begin{equation}\label{cond:evalue}
\lambda_1(\psi,u)<0<\lambda_2(\psi,u),
\end{equation}
which is, from (\ref{evalues})
\begin{equation}\label{cond:nonchar_normal}
u^2<\frac{1-\psi}{\psi}G'(1/\psi).
\end{equation}
Notice that since $\psi>0$, (\ref{cond:nonchar_normal}) implies that $G'(1/\psi)<0$ and moreover the hyperbolicity condition (\ref{eqn:hypcond1}).

As for the third condition (C3), the stable and unstable subspaces of $A^{-1}(\psi,u)$ is spanned by $R_{1}$ and $R_2$ (as defined in (\ref{rightevectors})) respectively. Hence the UKL condition (\ref{eqn:UKL}) is
satisfied when there exists a $\gamma>0$ such that
\begin{equation}\label{cond:UKL}
\psi\sqrt{\frac{1-\psi}{u^2+G'(1/\psi)}}\geq \gamma.
\end{equation}

In order to establish the wellposedness of the initial-boundary value problem (\ref{eqn:fixed2}) in some strong Sobolev space $H^m$ with $m>2$ being some integer, the data should satisfy the compatibility condition
\begin{eqnarray*}
\partial^p_t\psi^*(y,0)=\partial^p_t\psi(y,0), \quad \hbox{at }\ y=0,1,
\end{eqnarray*}
for all $p\in\{0,1,\ldots, m-1\}$. Since $\psi^*$ is some constant, the above condition is simply
\begin{equation}\label{eqn_compa}
\left\{\begin{array}{l}
\psi^*=\psi^0(1)=\psi^0(0)\\
\partial^p_t\psi(y,0)=0, \quad \hbox{at }\ y=0,1,\  \hbox{for all }\ p\in\{1,\ldots, m-1\}.
\end{array}\right.
\end{equation}

We can now state our local-wellposedness result.
\begin{theorem}\cite{BD}\label{thm:localwellposedness}
If $m>2$ is an integer, then for all $(\psi^0, u^0)\in
H^{m+1/2}([0,1])\times H^{m+1/2}([0,1])$ satisfying (\ref{cond:nonchar_normal}), (\ref{cond:UKL}), and the compatibility condition (\ref{eqn_compa}), there exists $T>0$ such that the problem
(\ref{eqn:fixed2}) admits a unique solution $u\in
H^m([0,1]\times[0,T])$.
\end{theorem}

\section{Long time existence of classical solutions}\label{sec_global}

The Cauchy problem of system (\ref{freebdry}) was discussed in
\cite{CZ}. The authors showed that the Cauchy problem is
$L^1$-stable, and the source term is merely weakly dissipative.
However, due to the result in \cite{Da}, the global existence of BV
solutions can still be obtained.

In concern of classical $C^1$ solutions to the initial-boundary value problem of (\ref{eqn:fixed2}), the short time existence can be established using the general approach introduced in Chapter 4 of \cite{LY}. To obtain large time of global $C^1$ solutions, one needs to control the $C^1$-norm of solutions. It turns out that if the systems exhibits strong enough dissipation or boundary damping then solutions can be extended globally in time (see, for instance \cite{Li}). However in our case, the above two types of damping are both weak. What we obtain is the large-time existence and uniqueness of $C^1$ solutions to the system (\ref{eqn:fixed2}), provided that the
data are close enough the equilibrium state $(\psi^*, 0)$.

Let $\eta=\psi-\psi^*$. To simplify the notation we still choose $(x,t)$ as the space-time variable. Then (\ref{eqn:fixed2}) becomes 
\begin{equation}\label{eqn:fixed3}
\left\{\begin{array}{l}
\left(\begin{array}{c}
                        \eta \\ u
                        \end{array}\right)_{t}+A(\eta,u)\left(\begin{array}{c}
                        \eta \\ u
                        \end{array}\right)_{x}+
P(\eta,u)=0,\hbox{ in } \ (0,1)\times(0,T)\\\\
B\left(\begin{array}{c}
                        \eta \\ u
                        \end{array}\right)=\left(\begin{array}{c}
                        0 \\ 0
                        \end{array}\right), \hbox{ at } x=0,1, \\\\
\left(\begin{array}{c}
                        \eta \\ u
                        \end{array}\right)\Big|_{t=0}=\left(\begin{array}{c}
                        \eta^0 \\ u^0
                        \end{array}\right)=\left(\begin{array}{c}
                        \psi^0-\psi^* \\ u^0
                        \end{array}\right), \hbox{ for } 0<x<1,
\end{array}\right.
\end{equation}
where
\begin{eqnarray*}
A(\eta,u)=\left(\begin{array}{cc}
                 \displaystyle -\frac{u}{(\eta+\psi^*)^2} & \displaystyle \frac{1-(\eta+\psi^*)}{\eta+\psi^*} \\\\
                 \displaystyle \frac{u^2+G'(1/(\eta+\psi^*))}{(\eta+\psi^*)^3} & \displaystyle  -\frac{u}{(\eta+\psi^*)^2}
                \end{array}\right),\ \ \
P(\eta,u)=\left(\begin{array}{c}
        0 \\ \displaystyle  \frac{\beta u(\eta+\psi^*)^2}{(\eta+\psi^*)-1}
        \end{array}\right).
\end{eqnarray*}

Now we state our main result.

\begin{theorem}\label{thm_LTexist}
Suppose that $\psi^*$ satisfies 
\begin{equation}\label{eqn_condpsi*}
\frac{1-\psi^*}{\psi^*}G'(1/\psi^*)>0
\end{equation}
and the $C^1$-compatibility conditions of the initial and boundary data
\begin{equation}\label{c1compa}
\left\{\begin{array}{l}
\eta^0(0)=\eta^0(1)=0,\\\\
\displaystyle -{u^0\over(\psi^*)^2}\eta^0_x+{1-\psi^*\over\psi^*}u^0_x=0.
\end{array}\right.
\end{equation}
hold.
Then for any $T_0>0$, there exists an $\varepsilon>0$ so that if  
\[
\|(\eta^0, u^0)\|_{C^1}\leq\varepsilon, \ \forall \ 0\leq x\leq 1,
\]
then system (\ref{eqn:fixed3}) admits a unique classical $C^1$ solution for $t\in[0,T_0)$.
\end{theorem}


\begin{proof}
As is pointed at the beginning of this section, the local-in-time wellposedness of $C^1$ solutions can be proved using the idea from \cite{LY}. For the time being, suppose that on the existence domain of $C^1$
solution $\eta(x,t), u(x,t)$ we have
\begin{equation}\label{unif_bd}
|(\eta,u)(x,t)|\leq \varepsilon_0,
\end{equation}
where $\varepsilon_0>0$ is a suitably small number so that
(\ref{cond:nonchar_normal}) 
is satisfied for
$|\eta|,|u|\leq\varepsilon_0$. Because of condition (\ref{eqn_condpsi*}) and continuity, such an $\varepsilon_0$ exists. 

To get the large-time existence of solutions, it suffices to prove that we can choose $\varepsilon_0>0$ small enough so that for any fixed $\varepsilon$ with $0<\varepsilon\leq\varepsilon_0$, there exists some $\delta=\delta(\varepsilon)>0$ small such that if
\begin{equation}\label{initialest}
\|(\eta^0, u^0)\|_{C^1}\leq\delta, \ \forall \ 0\leq x\leq 1
\end{equation}
holds, then on the whole existence domain of the $C^1$ solution $(\eta,u)$ we have
\begin{equation}\label{globalest}
\|(\eta, u)(\cdot, t)\|_{C^1}\leq\varepsilon, \ \ \forall \ t\geq0.
\end{equation}

First we perform a diagonalization. Let
\begin{eqnarray*}
\left\{\begin{array}{l}
        v_i=L_i(\eta,u)(\eta,u)^T \ \ (i=1,2),\\
        w_i=L_i(\eta,u)(\partial_x\eta, \partial_xu)^T \ \ (i=1,2),
       \end{array}\right.
\end{eqnarray*}
where $L_i$ is the $i$-th left eigenvector defined in
(\ref{leftevectors}) with $\psi$ replaced by $\eta+\psi^*$. It is
not hard to see that $v=(w_1,v_2)$ and $w=(w_1,w_2)$ satisfy the
following system of diagonal form (cf. Chapter 3 in \cite{Li})
\begin{equation}\label{approx_diag}
\left\{\begin{array}{l}
\displaystyle \partial_t v_i+\lambda_i\partial_x v_i+\kappa(v_1+v_2)=\sum_{j,k=1}^2
c_{ijk}v_jv_k+\sum_{j,k=1}^2 d_{ijk}v_jw_k\\\\
\displaystyle \partial_t w_i+\lambda_i\partial_x w_i+\kappa(w_1+w_2)=\sum_{j,k=1}^2
\bar{c}_{ijk}w_jv_k+\sum_{j,k=1}^2 \bar{d}_{ijk}v_jw_k
\end{array}\right.
\end{equation}
for $i=1,2$, where
$$\kappa=\frac{\beta(\psi^*)^2}{2(\psi^*-1)}>0,$$
$c_{ijk}, d_{ijk}, \bar{c}_{ikj}$ and $\bar{d}_{ijk}$ are continuous
functions of $(\eta,u)$, $\lambda_i(t,x)=\lambda_i(\eta,u)$ (cf. (\ref{evalues}) with $\psi$ replaced by $\eta+\psi^*$). Denote the quadratic parts by 
\begin{eqnarray*}
&&Q_i(v,w)=\sum_{j,k=1}^2
c_{ijk}v_jv_k+\sum_{j,k=1}^2 d_{ijk}v_jw_k\ \ (i=1,2),\\
&&\bar{Q}_i(v,w)=\sum_{j,k=1}^2
\bar{c}_{ijk}w_jv_k+\sum_{j,k=1}^2 \bar{d}_{ijk}v_jw_k\ \
(i=1,2).
\end{eqnarray*}

The initial condition for $v$ and $w$ is obviously
\begin{eqnarray*}
\left\{\begin{array}{l}
        v_i|_{t=0}=v_i^0=L_i(\eta^0,u^0)(\eta^0,u^0)^T \ \ (i=1,2),\\
        w_i|_{t=0}=w_i^0=L_i(\eta^0,u^0)(\partial_x\eta^0, \partial_xu^0)^T \ \
        (i=1,2).
       \end{array}\right.
\end{eqnarray*}
Boundary condition for $v$:
\begin{eqnarray*}
\hbox{At }x=0,\ 1: \hspace{.2in} B\left(\begin{array}{c}
        L_1\\L_2
       \end{array}\right)^{-1}v=\left(\begin{array}{c}
        0\\0
       \end{array}\right).
\end{eqnarray*}
Thus
\begin{equation}\label{bdrycond_v}
\hbox{At }x=0,\ 1: \hspace{.2in} v_1=v_2.
\end{equation}
Differentiating with respect to $t$, we get the boundary condition
for $w$:
\begin{eqnarray*}
\left(\begin{array}{c}
        0\\0
       \end{array}\right)&=B\left(\begin{array}{c}
        \eta\\u
       \end{array}\right)_t=B\Big[-A(\eta,u)\left(\begin{array}{c}
                        \eta \\ u
                        \end{array}\right)_{x}-P(\eta,u)\Big]=-BA(\eta,u)\left(\begin{array}{c}
                        \eta \\ u
                        \end{array}\right)_{x}\\
&=-BA(\eta,u)\left(\begin{array}{c}
        L_1\\L_2
       \end{array}\right)^{-1}\left(\begin{array}{c}
                        w_1 \\ w_2
                        \end{array}\right).
\end{eqnarray*}
Therefore
\begin{equation}\label{bdrycond_w}
\hbox{At }x=0,\ 1: \hspace{.2in} w_1=w_2.
\end{equation}

Therefore in order to prove (\ref{globalest}), we only need to show that the same estimate holds for $\|(v,w)\|_{C^0}$:
\begin{equation}\label{globalest1}
|(v,w)(x,t)|\leq\varepsilon.
\end{equation}
provided that the initial data is small:
\begin{equation}\label{initialest1}
|(v^0, w^0)|\leq\delta, \ \forall \ 0\leq x\leq 1.
\end{equation}
From (\ref{unif_bd}), we may also assume that on the whole existence domain of the $C^0$ solution $(v,w)$,
\begin{equation}\label{unif_bd1}
|(v,w)(x,t)|\leq \varepsilon_0.
\end{equation}

Let
\begin{eqnarray*}
&&\lambda_{\min}=\min\{|\lambda_i(t,x)|:\ 0\leq t\leq T, 0\leq x\leq 1, i=1,2\},\\
&&\lambda_{\max}=\max\{|\lambda_i(t,x)|:\ 0\leq t\leq T, 0\leq x\leq 1, i=1,2\},\\
&&T_1=1/\lambda_{\max},\hspace{.2in} T_2=1/\lambda_{\min}.
\end{eqnarray*}
Since $\varepsilon_0$ is chosen so that the hyperbolicity condition (\ref{cond:nonchar_normal}) holds, we see that $\lambda_{\min}<0$ and $\lambda_{\max}>0$ and hence $T_1$ and $T_2$ are well-defined.

By continuity we know that we can certainly pick some $\delta>0$
small such that (\ref{globalest1}) holds on some time domain. Hence
to prove (\ref{globalest1}), it is only necessary to show that there
exists some $\varepsilon_0>0$ so small that for any fixed $T>0$, if
(\ref{globalest1}) holds on the domain $D(T)=\{(t,x):\ 0\leq t\leq
T, 0\leq x\leq 1\}$, then it still holds on $D(T+T_1)$, provided
that the $C^0$ solution $(v,w)$ exists on such domain.

For this purpose, let
\begin{eqnarray}\label{def:max}
&&V_i(t)=\max_{0\leq x\leq1}|v_i(x,t)|,\ \ W_i(t)=\max_{0\leq
x\leq1}|w_i(x,t)|, \ \ i=1,2, \nonumber\\
&&U_1(t)=\max\{V_1(t),V_2(t)\},\ \ \ \ \ \
U_2(t)=\max\{W_1(t),W_2(t)\}.
\end{eqnarray}

Suppose that the $C^0$ solution $(v,w)$ exists on $D(T+T_1)$ and let
$\xi=f_i(\tau; x,t)$ be the $i$-th characteristic passing through a
point $(x,t)\in D(T+\bar{T})$ with $T\leq t\leq T+T_1$. Then
\begin{equation}\label{chareqn}
\left\{\begin{array}{l}
\displaystyle \frac{d}{d\tau}f_i(\tau; x,t)=\lambda_i(\tau, f_i(\tau;x,t)),\\\\
\tau=t:\ \ f_i(t; x,t)=x.
\end{array}\right.
\end{equation}

From (\ref{cond:evalue}) we know that for $v_1(x,t)$ there are two
possibilities:

(1) The first characteristic $\xi=f_1(\tau; x, t)$ intersects the
interval $[0,1]$ on $x$-axis with the intersection point $(f_1(0;
x,t),0)$, see {Fig. 3}.

\begin{figure}[h]
\begin{center}
\includegraphics[scale=.4]{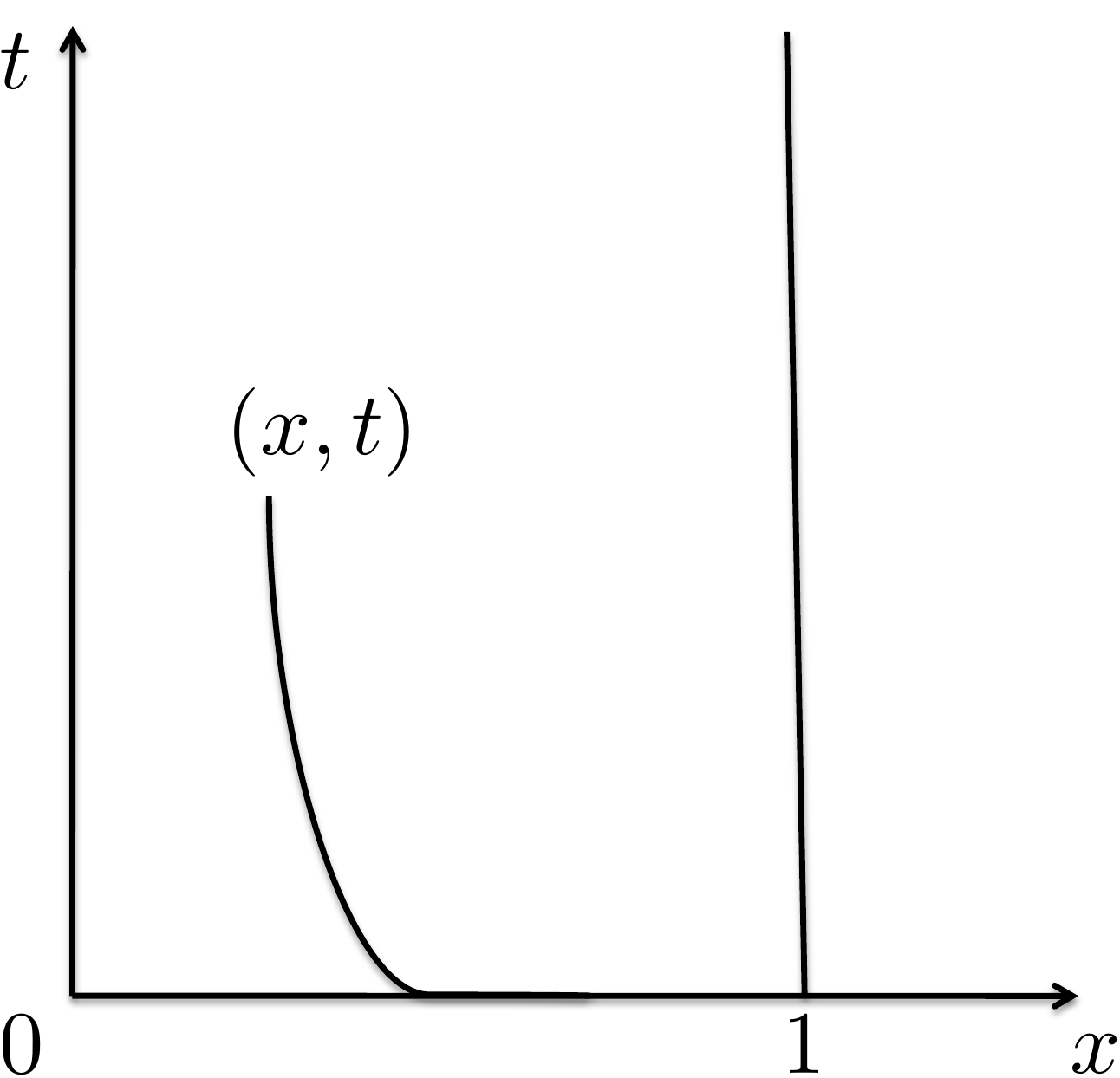}
\end{center}
\caption{}
\end{figure}

Integrating the first equation in (\ref{approx_diag}) along the
first characteristic $\xi=f_1(\tau; x, t)$ we get
\begin{eqnarray}
v_1(x,t)=&&\ e^{-\kappa t}v_1^0(f_1(0; x,t))-\int^t_0\kappa e^{-\kappa(t-\tau)}v_2(f_1(\tau; x,t),\tau)d\tau\nonumber\\
&&-\int^t_0e^{-\kappa(t-\tau)}Q_1(v,w)(f_1(\tau; x,t),\tau)d\tau.
\end{eqnarray}
From (\ref{initialest1}) and since (\ref{globalest1})  holds on $D(T)$, we have
\begin{eqnarray}\label{est_case1}
e^{\kappa t}|v_1(x,t)|&&\leq \ \delta+(e^{\kappa
T}-1)\varepsilon+C\varepsilon^2+\int^t_T\kappa
e^{\kappa\tau}V_2(\tau)d\tau
+C\int^t_Te^{\kappa\tau}\sum^2_{i=1}V_i^2(\tau)+W_i^2(\tau)d\tau \nonumber\\
&&\leq\ \delta+(e^{\kappa T}-1)\varepsilon+C\varepsilon^2+\int^t_T\kappa
e^{\kappa\tau}U_1(\tau)d\tau
+C\int^t_Te^{\kappa\tau}(U_1^2(\tau)+U_2^2(\tau))d\tau,
\end{eqnarray}
where $V_i, W_i$ and $U_i$ were defined in (\ref{def:max}).

(2) The first characteristic $\xi=f_1(\tau; x, t)$ intersects the
boundary $x=1$ at point $(1,\tau_1(x,t))$ where $(\tau_1(x,t))$
satisfies
\[ f_1(\tau_1(x,t);x,t)=1. \]
Obviously we have
\begin{equation}\label{timeest1}
T_2\geq t-\tau_1(x,t)\geq 0.
\end{equation}

In case (2), for the second characteristic $\xi=f_2(\tau;
1,\tau_1(x,t))$ passing through $(1,\tau_1(x,t))$, there are still
two possibilities:

(2a) This second characteristic intersects the interval $[0,1]$ on
the $x$-axis with the intersection point $(f_2(0;1,\tau_1(x,t)),0)$,
see {Fig. 4a}.
\begin{figure}[h]
\begin{minipage}{.48\textwidth}
\begin{center}
\includegraphics[scale=.4]{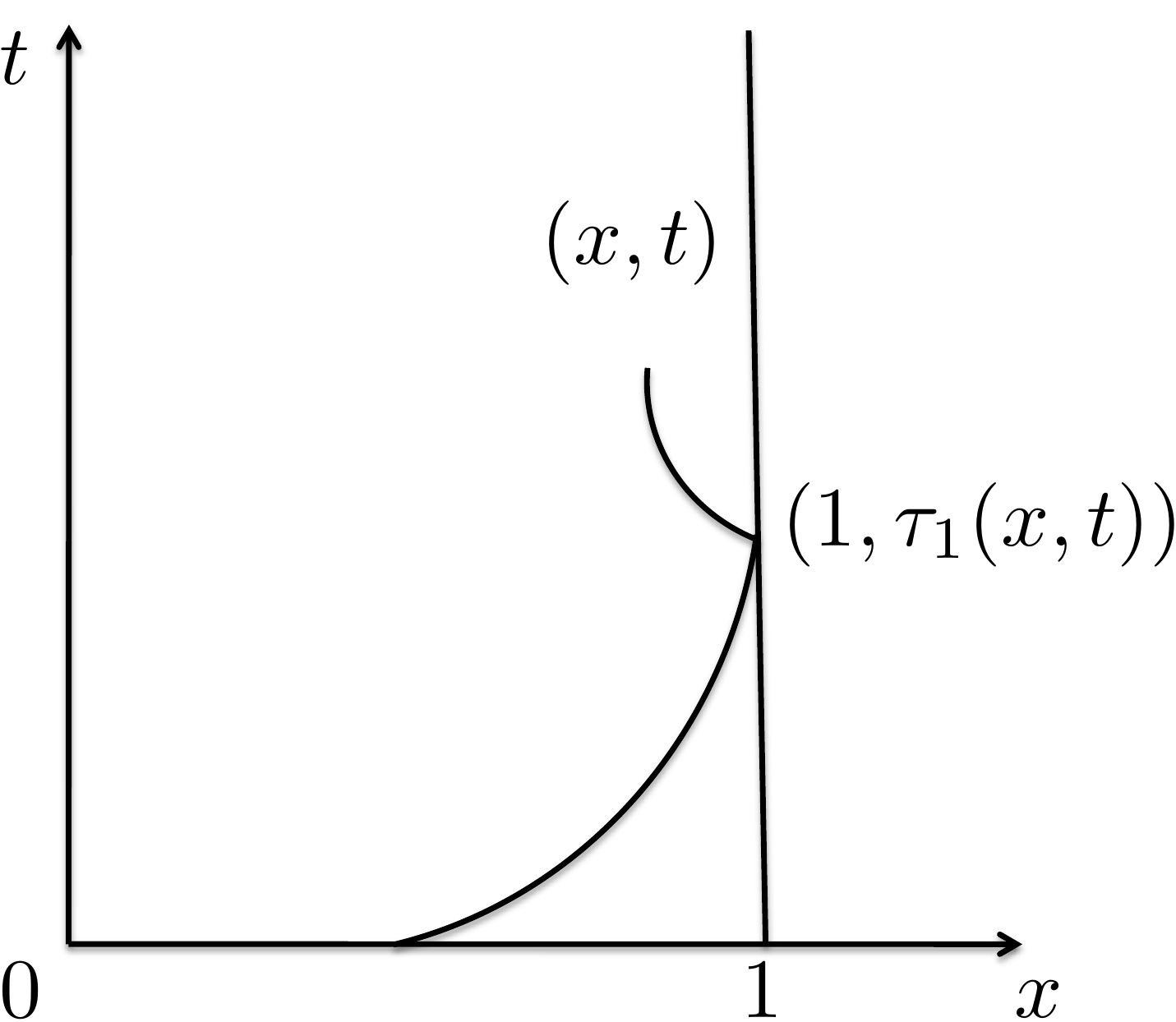}
{(a)}
\end{center}
\end{minipage}
\begin{minipage}{.48\textwidth}
\begin{center}
\includegraphics[scale=.4]{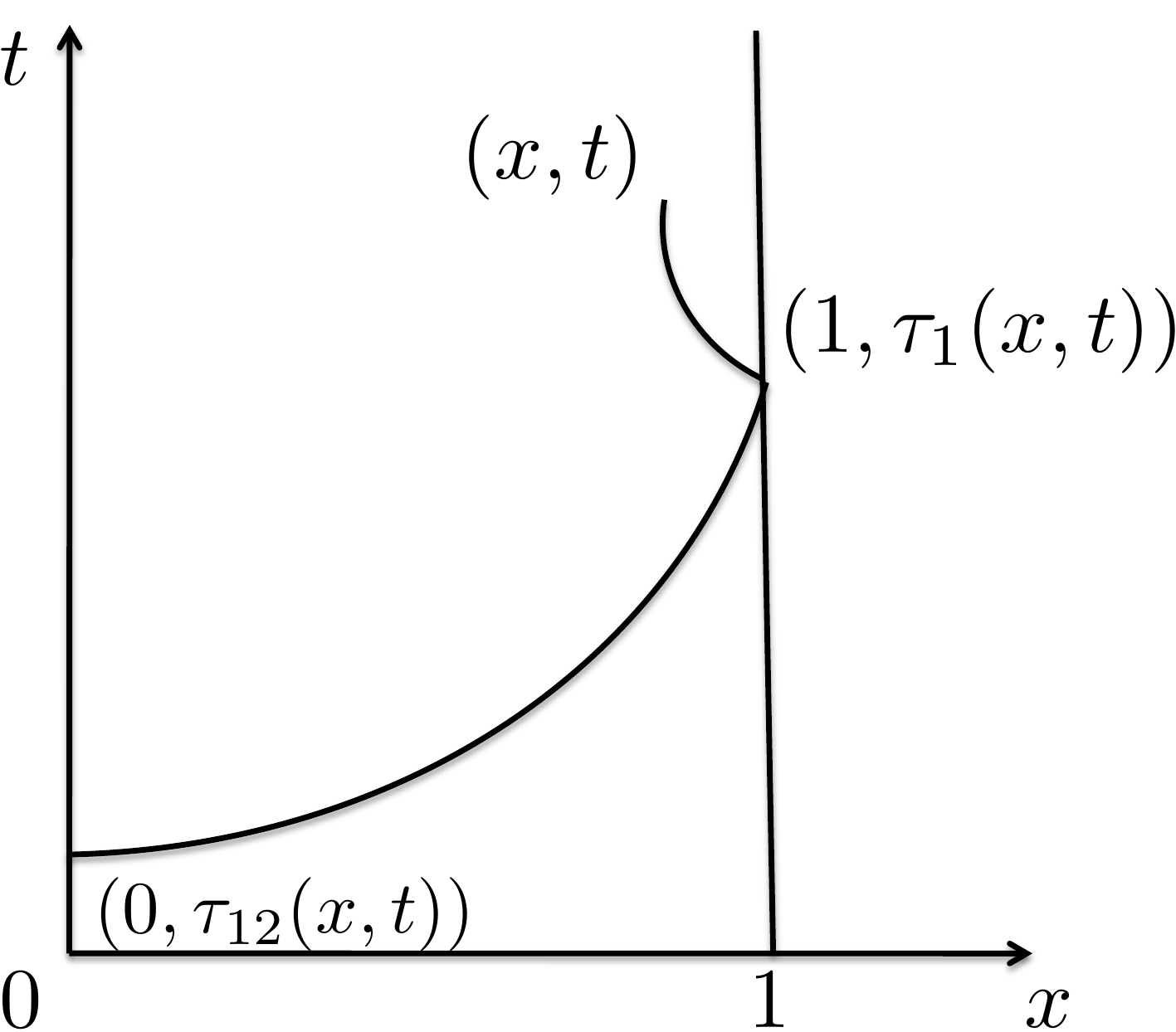}
{(b)}
\end{center}
\end{minipage}
\caption{}
\end{figure}

(2b) This second characteristic intersects the boundary $x=0$ at the
intersection point $(0, \tau_{12}(x,t))$ (see {Fig. 4b}),
where
\[ \tau_{12}(x,t)=\tau_2(\tau_1(x,t),1), \]
in which, $\tau_2(x,t)$ stands for the $t$-coordinate of the
intersection point of the second characteristic $\xi=f_2(\tau;x,t)$
passing through a point $(x,t)$ with the boundary $x=0$:
\[ f_2(\tau_2(x,t);x,t)=0. \]
We have
\begin{equation}\label{timeest2}
T_2\geq t-\tau_2(x,t)\geq 0.
\end{equation}

Therefore
\[ T_2\geq \tau_1(x,t)-\tau_{12}(x,t)\geq0. \]
Noting (\ref{timeest1}) we get
\[ 2T_2\geq t-\tau_{12}(x,t)\geq T_1. \]
Therefore for $T\leq t\leq T+T_1$,
\begin{equation}\label{timeest3}
0\leq \tau_{12}(x,t)\leq T,\ \ \ \ T-\tau_{12}(x,t)\leq 2T_2.
\end{equation}

In case (2a), using boundary condition (\ref{bdrycond_v}) we get
\begin{eqnarray*}
e^{\kappa t}v_1(x,t)&&=\
e^{\kappa\tau_1}v_1(1,\tau_1)-\int_{\tau_1}^{t}\kappa
e^{\kappa\tau}v_2(f_1(\tau;x,t),\tau)d\tau-\int^t_{\tau_1}e^{\kappa\tau}Q_1(v,w)(
f_1(\tau; x,t),\tau)d\tau\\
&&=\ e^{\kappa\tau_1}v_2(1,\tau_1)-\int_{\tau_1}^{t}\kappa
e^{\kappa\tau}v_2(f_1(\tau;x,t),\tau)d\tau-\int^t_{\tau_1}e^{\kappa\tau}Q_1(v,w)(
f_1(\tau; x,t),\tau)d\tau.
\end{eqnarray*}

Integrating the second equation in (\ref{approx_diag}) along the
second characteristic we have
\begin{eqnarray*}
e^{\kappa\tau_1}v_2(1,\tau_1)=&&\
v^0_2(f_2(0;1,\tau_1))-\int^{\tau_1}_0\kappa
e^{\kappa\tau}v_1(f_2(\tau;1,\tau_1),\tau)d\tau\\
&&\ -\int^{\tau_1}_0e^{\kappa\tau}Q_2(v,w)( f_2(\tau;
1,\tau_1),\tau)d\tau.
\end{eqnarray*}

Combing the above equalities we get
\begin{eqnarray*}
e^{\kappa t}v_1(x,t)=&&\ v^0_2(f_2(0;1,\tau_1))-\int^{\tau_1}_0\kappa
e^{\kappa\tau}v_1(f_2(\tau;1,\tau_1),\tau)d\tau-\int_{\tau_1}^{t}\kappa
e^{\kappa\tau}v_2(f_1(\tau;x,t),\tau)d\tau\\
&&\ -\int^{\tau_1}_0e^{\kappa\tau}Q_2(v,w)( f_2(\tau;
1,\tau_1),\tau)d\tau-\int^t_{\tau_1}e^{\kappa\tau}Q_1(v,w)(
f_1(\tau; x,t),\tau)d\tau.
\end{eqnarray*}

We have two cases:

(i) $\tau_1\leq T$. Then it is easy to see that we have the same
estimate as (\ref{est_case1}).\\

(ii) $\tau>T$. Then
\begin{eqnarray}\label{est_case2a}
e^{\kappa t}|v_1(x,t)|&&\leq \ \delta+(e^{\kappa
T}-1)\varepsilon+C\varepsilon^2+\int^{\tau_1}_T\kappa
e^{\kappa\tau}V_1(\tau)d\tau+\int^t_{\tau_1}\kappa
e^{\kappa\tau}V_2(\tau)d\tau \nonumber\\
&&\ \ \
+C\int^t_Te^{\kappa\tau}\sum^2_{i=1}V_i^2(\tau)+W_i^2(\tau)d\tau\nonumber\\
&&\leq\ \delta+(e^{\kappa T}-1)\varepsilon+C\varepsilon^2+\int^t_T\kappa
e^{\kappa\tau}U_1(\tau)d\tau
+C\int^t_Te^{\kappa\tau}(U_1^2(\tau)+U_2^2(\tau))d\tau,
\end{eqnarray}
which is again, the same as (\ref{est_case1}).\\

In case (2b), using the similar idea we obtain
\begin{eqnarray*}
e^{\kappa t}v_1(x,t)&&=\
e^{\kappa\tau_1}v_1(1,\tau_1)-\int_{\tau_1}^{t}\kappa
e^{\kappa\tau}v_2(f_1(\tau;x,t),\tau)d\tau-\int^t_{\tau_1}e^{\kappa\tau}Q_1(v,w)(
f_1(\tau; x,t),\tau)d\tau\\
&&=\ e^{\kappa\tau_1}v_2(1,\tau_1)-\int_{\tau_1}^{t}\kappa
e^{\kappa\tau}v_2(f_1(\tau;x,t),\tau)d\tau-\int^t_{\tau_1}e^{\kappa\tau}Q_1(v,w)(
f_1(\tau; x,t),\tau)d\tau.\\
&&=\
e^{\kappa\tau_{12}}v_2(0,\tau_{12})-\int^{\tau_1}_{\tau_{12}}\kappa
e^{\kappa\tau}v_1(f_2(\tau;1,\tau_1),\tau)d\tau-\int_{\tau_1}^{t}\kappa
e^{\kappa\tau}v_2(f_1(\tau;x,t),\tau)d\tau\\
&&\ \ \  -\int^{\tau_1}_{\tau_{12}}e^{\kappa\tau}Q_2(v,w)( f_2(\tau;
1,\tau_1),\tau)d\tau-\int^t_{\tau_1}e^{\kappa\tau}Q_1(v,w)(
f_1(\tau; x,t),\tau)d\tau.
\end{eqnarray*}

From (\ref{timeest3}) we know that $0\leq \tau_{12}\leq T$. Hence we
can again split case (2b) in the following two possibilities:

(i) $\tau_1\leq T$. Then the estimate for $|v_1(x,t)|$ is
\begin{eqnarray}\label{est_case2bi}
e^{\kappa t}|v_1(x,t)|&&\leq\ \delta+(e^{\kappa
T}-1)\varepsilon+C\varepsilon^2+\int^t_T\kappa
e^{\kappa\tau}V_2(\tau)d\tau+C\int^t_Te^{\kappa\tau}\sum^2_{i=1}V_i^2(\tau)+W_i^2(\tau)d\tau \nonumber\\
&&\leq\ \delta+(e^{\kappa T}-1)\varepsilon+C\varepsilon^2+\int^t_T\kappa
e^{\kappa\tau}U_1(\tau)d\tau
+C\int^t_Te^{\kappa\tau}(U_1^2(\tau)+U_2^2(\tau))d\tau.
\end{eqnarray}

(ii) $\tau_1>T$. Then
\begin{eqnarray}\label{est_case2bii}
e^{\kappa t}|v_1(x,t)|&&\leq\ \delta+(e^{\kappa
T}-1)\varepsilon+C\varepsilon^2+\int^{\tau_1}_T\kappa
e^{\kappa\tau}V_1(\tau)d\tau+\int^t_{\tau_1}\kappa
e^{\kappa\tau}V_2(\tau)d\tau \nonumber\\
&&\ \ \ +C\int^t_Te^{\kappa\tau}\sum^2_{i=1}V_i^2(\tau)+W_i^2(\tau)d\tau \nonumber\\
&&\leq\ \delta+(e^{\kappa T}-1)\varepsilon+C\varepsilon^2+\int^t_T\kappa
e^{\kappa\tau}U_1(\tau)d\tau
+C\int^t_Te^{\kappa\tau}(U_1^2(\tau)+U_2^2(\tau))d\tau.
\end{eqnarray}

Hence we always have
\begin{equation}\label{est_final_1}
e^{\kappa t}|v_1(x,t)| \leq\ \delta+(e^{\kappa
T}-1)\varepsilon+C\varepsilon^2+\int^t_T\kappa
e^{\kappa\tau}U_1(\tau)d\tau
+C\int^t_Te^{\kappa\tau}(U_1^2(\tau)+U_2^2(\tau))d\tau.
\end{equation}

In the same way, we obtain similar estimates for $|v_2|, |w_1|$ and
$|w_2|$. Therefore we have
\begin{equation}\label{est_final}
\left\{\begin{array}{l}
e^{\kappa t}U_1(x,t) \leq\ \delta+(e^{\kappa
T}-1)\varepsilon+C\varepsilon^2+\int^t_T\kappa
e^{\kappa\tau}U_1(\tau)d\tau
+C\int^t_Te^{\kappa\tau}(U_1^2(\tau)+U_2^2(\tau))d\tau, \\\\
e^{\kappa t}U_2(x,t) \leq\ \delta+(e^{\kappa
T}-1)\varepsilon+C\varepsilon^2+\int^t_T\kappa
e^{\kappa\tau}U_2(\tau)d\tau
+C\int^t_Te^{\kappa\tau}(U_1^2(\tau)+U_2^2(\tau))d\tau.
\end{array}\right.
\end{equation}

Now let
\begin{equation}\label{sub}
\left\{\begin{array}{l}
e^{\kappa t}Y_1(x,t) =\ \delta+(e^{\kappa
T}-1)\varepsilon+C\varepsilon^2+\int^t_T\kappa
e^{\kappa\tau}Y_1(\tau)d\tau
+C\int^t_Te^{\kappa\tau}(Y_1^2(\tau)+Y_2^2(\tau))d\tau, \\\\
e^{\kappa t}Y_2(x,t) =\ \delta+(e^{\kappa
T}-1)\varepsilon+C\varepsilon^2+\int^t_T\kappa
e^{\kappa\tau}Y_2(\tau)d\tau
+C\int^t_Te^{\kappa\tau}(Y_1^2(\tau)+Y_2^2(\tau))d\tau.
\end{array}\right.
\end{equation}
Then obviously, $U_i\leq Y_i$ for $i=1,2$ and $Y_i$ satisfies
\begin{equation}\nonumber
\left\{\begin{array}{l}
\displaystyle \frac{d}{dt}Y_i(t)=C(Y_1^2+Y_2^2),\\\\
Y_i(T)=e^{-\kappa T}\Big[\delta+(e^{\kappa
T}-1)\varepsilon+C\varepsilon^2\Big].
\end{array}\right.
\end{equation}

Therefore $\forall\
t\in [T,T+T_1],$
\begin{equation}\label{estY1}
Y(t)\leq e^{-\kappa T}\Big[\delta+(e^{\kappa
T}-1)\varepsilon+C\varepsilon^2\Big]+C\int^t_TY(\tau)^2d\tau,
\end{equation}
where $Y(t)=(Y_1(t),Y_2(t))^T$. More precisely, we have $\forall\
t\in [T,T+T_1]$,
\begin{equation}\label{estY2}
Y(t)\leq \frac{e^{-\kappa T}\Big[\delta+(e^{\kappa
T}-1)\varepsilon+C\varepsilon^2\Big]}{1-Ce^{-\kappa
T}\Big[\delta+(e^{\kappa T}-1)\varepsilon+C\varepsilon^2\Big](t-T)}.
\end{equation}
Therefore for any given $T_0>0$, we can choose $0<\varepsilon_0<1/(2C)$
sufficiently small and independent of $T$. For any fixed $\varepsilon$
($0<\varepsilon\leq\varepsilon_0$), there exists
$\delta=\delta(\varepsilon)>0$ also independent of $T$, such that
\begin{equation}\label{estY3}
Y(t)\leq \frac{e^{-\kappa T}\delta+(1-e^{-\kappa
T}/2)\varepsilon}{1-C\Big[e^{-\kappa T}\delta+(1-e^{-\kappa
T}/2)\varepsilon\Big]T_1}\leq\varepsilon,\ \forall\ t\in [T,T+T_1].
\end{equation}
Therefore we may choose $\varepsilon_0$ so small that
$2C\varepsilon_0<1$ and $CT_1\varepsilon_0<1/(4e^{\kappa T_0}-1)$. Then
choose $\delta\leq\varepsilon/4$. In this way, we further have $T_0 \sim |\log\varepsilon|$. This proves the theorem.

\end{proof}

\section{Conclusion} \label{sec_Con}
We studied the free boundary problem of  swelling of a gel with inertia and elastic effects dominating the viscous ones, in one space dimension. 
 This corresponds to  early dynamics regime for  many polymeric gels
and also describes the main dynamics for polyssaccharide gel networks. We transform the free boundary problem of  the weakly dissipative hyperbolic
governing equation to one fixed domain. As a result  of requiring the system to admit one dimensional solutions, it turns out that  the polymer 
volume fraction at the 
interface with the surrounding fluid  is fully determined, if we neglect fluid inertia in the boundary conditions.

We find necessary and sufficient conditions for the system to be hyperbolic, this requiring a monotonicity property of a one-dimensional stress. 
In current work, we are studying the case that such condition fails, and found that it may be indicative of 
 two possible phenomena, one being {\it de-swelling}, that is the interface will move backwards, and the second one may be related to the gel locally
 experiencing  a volume phase transition. This is a phenomenon analogous to the volume transition in polyelectrolyte gels that occur when a relevant 
ion concentration reaches a critical value.  

\vspace{.3cm}

{\bf \noindent Acknowledgement.} M.C. Carme was supported in part by NSF grants DMS-0909165. R.M. Chen was supported in part by NSF grants DMS-0908663.



\end{document}